\documentclass[a4paper, 11pt]{amsart}
\usepackage{mathptmx,amssymb,amscd,latexsym, eulervm}
\usepackage{amsmath}
\usepackage{amsthm}
\usepackage{mathdots}
\usepackage[colorlinks=true,citecolor=violet,linkcolor=blue,urlcolor=blue]{hyperref}
\usepackage[dvipsnames]{xcolor}
\usepackage[onehalfspacing]{setspace}
\usepackage{tabularx}
\usepackage{amsfonts}
\usepackage{paralist}
\usepackage{aliascnt}
\usepackage{amscd}
\usepackage{blkarray}
\usepackage{mathbbol}
\usepackage{setspace}
\usepackage{needspace}
\usepackage[inner=2.4cm,outer=2.4cm, bottom=3.2cm]{geometry}
\usepackage{tikz, tikz-cd}
\usepackage{calligra,mathrsfs}

\usepackage{tikz}
\usetikzlibrary{matrix}
\usetikzlibrary{arrows,calc}
\allowdisplaybreaks

\AtBeginDocument{%
	\def\MR#1{}
}

\makeatletter
\@namedef{subjclassname@2020}{%
	\textup{2020} Mathematics Subject Classification}
\makeatother

\newcommand{\RR}{\mathbb{R}}

\newcommand{\kk}{\mathbb{k}}
\newcommand{\KK}{\mathbb{K}}

\newcommand{\bu}{\mathbf{u}}

\newcommand{\rk}{\text{\rm rk}}

\newcommand{\sP}{\mathscr{P}}
\newcommand{\cone}{{\rm Cone}}
\newcommand{\Cone}{{\rm Cone}}
\newcommand{\Rec}{\operatorname{Rec}}

\newcommand{\CC}{\mathbb{C}}

\newcommand{\NN}{\normalfont\mathbb{N}}
\newcommand{\ZZ}{\mathbb{Z}}

\newcommand{\PP}{{\normalfont\mathbb{P}}}

\newcommand{\QQ}{\mathbb{Q}}

\newcommand{\bn}{{\normalfont\mathbf{n}}}
\newcommand{\bm}{{\normalfont\mathbf{m}}}

\newcommand{\TMsupp}{{\normalfont\text{TMSupp}}}

\newcommand{\ttt}{{\normalfont\mathbf{t}}}

\newcommand{\Ker}{\normalfont\text{Ker}}

\newcommand{\IM}{\normalfont\text{Im}}

\newcommand{\ee}{{\normalfont\mathbf{e}}}
\newcommand{\Hom}{\normalfont\text{Hom}}

\newcommand{\OO}{\mathscr{O}}

\newcommand{\LL}{\mathscr{L}}

\newcommand{\st}{\smallfrown_{\rm st}}

\newcommand{\HH}{\normalfont\text{H}}

\newcommand{\Spec}{\normalfont\text{Spec}}

\newcommand{\msupp}{\normalfont\text{MSupp}}

\newcommand{\NR}{N_{\RR}}
\newcommand{\TL}{T_L}
\newcommand{\bdeq}{\sim_{\mathrm{bd}}}

\DeclareMathOperator{\Trop}{Trop}
\DeclareMathOperator{\Pic}{Pic}
\DeclareMathOperator{\Relint}{{\rm Relint}}
\DeclareMathOperator{\Span}{{\rm Span}}
\DeclareMathOperator{\Star}{{\rm Star}}
\DeclareMathOperator{\MW}{{\rm MW}}

\def\fu{\mathbf{u}}
\def\f0{\mathbf{0}}

\def\fv{\mathbf{v}}

\def\1{\mathbf{1}}

\def\ba{\mathbf{a}}

\def\bw{\mathbf{w}}
\def\bv{\mathbf{v}}

\def\ba{\mathbf{a}}

\newtheorem{theorem}{Theorem}[section]

\newtheorem{headthm}{Theorem}

\newaliascnt{headcor}{headthm}
\newtheorem{headcor}[headcor]{Corollary}
\aliascntresetthe{headcor}

\newaliascnt{headconj}{headthm}

\aliascntresetthe{headconj}

\newaliascnt{corollary}{theorem}
\newtheorem{corollary}[corollary]{Corollary}
\aliascntresetthe{corollary}

\newaliascnt{claim}{theorem}

\aliascntresetthe{claim}

\newaliascnt{lemma}{theorem}
\newtheorem{lemma}[lemma]{Lemma}
\aliascntresetthe{lemma}

\newaliascnt{conjecture}{theorem}

\aliascntresetthe{conjecture}

\newaliascnt{proposition}{theorem}
\newtheorem{proposition}[proposition]{Proposition}
\aliascntresetthe{proposition}

\theoremstyle{definition}
\newaliascnt{definition}{theorem}
\newtheorem{definition}[definition]{Definition}
\aliascntresetthe{definition}

\newaliascnt{notation}{theorem}

\aliascntresetthe{notation}

\newaliascnt{example}{theorem}
\newtheorem{example}[example]{Example}
\aliascntresetthe{example}

\newaliascnt{examples}{theorem}

\aliascntresetthe{examples}

\newaliascnt{remark}{theorem}
\newtheorem{remark}[remark]{Remark}
\aliascntresetthe{remark}

\newaliascnt{question}{theorem}
\newtheorem{question}[question]{Question}
\aliascntresetthe{question}

\newaliascnt{questions}{theorem}

\aliascntresetthe{questions}

\newaliascnt{problem}{theorem}

\aliascntresetthe{problem}

\newaliascnt{construction}{theorem}

\aliascntresetthe{construction}

\newaliascnt{setup}{theorem}

\aliascntresetthe{setup}

\newaliascnt{algorithm}{theorem}

\aliascntresetthe{algorithm}

\newaliascnt{observation}{theorem}

\aliascntresetthe{observation}

\newaliascnt{defprop}{theorem}

\aliascntresetthe{defprop}

\newaliascnt{fact}{theorem}

\aliascntresetthe{fact}

\DeclareFontFamily{OT1}{pzc}{}
\DeclareFontShape{OT1}{pzc}{m}{it}{<-> s * [1.100] pzcmi7t}{}
\DeclareMathAlphabet{\mathchanc}{OT1}{pzc}{m}{it}

\DeclareMathOperator{\cl}{{\rm cl}}

\def\equationautorefname~#1\null{(#1)\null}
\def\sectionautorefname~#1\null{Section #1\null}
\def\subsectionautorefname~#1\null{\S #1\null}

\def\surjects{\twoheadrightarrow}

\title{When are tropical multidegrees positive?}
\author{Yairon Cid-Ruiz}
\address{Department of Mathematics, North Carolina State University, Raleigh, NC 27695, USA}
\email{ycidrui@ncsu.edu}

\date{\today}

\hypersetup{
	pdftitle={When are tropical multidegrees positive?},
	pdfauthor={Yairon Cid-Ruiz},
	pdfkeywords={tropical multidegrees, stable intersections, polymatroids, Lorentzian polynomials}
}

\begin{document}

\begin{abstract}
We study the positivity of the tropical multidegrees of a tropical variety contained in a product of real vector spaces. 
These multidegrees are obtained by stably intersecting the tropical variety with pullbacks of positive tropical divisors.
We introduce projection-purity and facet-selectability, two conditions under which positivity is determined by the dimensions of the natural projections, and the support of the tropical multidegrees is precisely the set of lattice points of a polymatroid base polytope. 
This extends He’s theorem for translation-admissible tropical varieties. 
We also show that these conditions alone do not force the corresponding tropical volume polynomial to be Lorentzian. 
By contrast, for the augmented Bergman fan of any polymatroid, the positive multidegrees are supported precisely on the lattice points of the polymatroid base polytope, and the tropical volume polynomial is Lorentzian for every sequence of positive tropical divisors.
\end{abstract}

	\maketitle

	\tableofcontents

\section{Introduction}

The degree of a projective variety is one of its most fundamental numerical invariants.
In a product of projective spaces, it is replaced by a collection of intersection numbers, called \emph{multidegrees}, recording the geometry of the variety with respect to the different factors.
Going back to seminal work of van der Waerden \cite{VAN_DER_WAERDEN}, multidegrees play a fundamental role in subjects ranging from algebraic geometry and commutative algebra to combinatorics, convex geometry, Schubert calculus, and algebraic statistics (see, e.g., \cite{Bhattacharya,HERMANN_MULTIGRAD,trung2001positivity,KNUTSON_MILLER_SCHUBERT,TRUNG_VERMA,STURMFELS_UHLER,Huh12,HUH_STURMFELS_LIKELIHOOD,EXPONENTIAL_VARIETIES,CCLMZ,cidruiz2021mixed,michalek2020maximum,CCC,manivel2020complete}).
More precisely, let $\kk$ be a field and 
$
X\subset \PP:=\PP_\kk^{m_1}\times_\kk\cdots\times_\kk \PP_\kk^{m_p}
$
be an irreducible multiprojective variety of dimension $d$, and let $H_i$ denote the pullback of the hyperplane class from the $i$-th factor $\PP_\kk^{m_i}$.
For every $\bn=(n_1,\ldots,n_p)\in\NN^p$ with $|\bn|=n_1+\cdots+n_p=d$, the corresponding \emph{multidegree} is
$$
\deg_\PP^\bn(X)
\;:=\;
\int_\PP H_1^{n_1}\cdots H_p^{n_p} \cdot \left[X\right].
$$
These numbers are nonnegative, and hence a fundamental question arises: \emph{when are multidegrees
positive?}
For $I\subseteq[p]=\{1,\ldots,p\}$, let
$
\Pi_I:\PP\longrightarrow\prod_{i\in I}\PP_\kk^{m_i}
$
be the natural projection.
Castillo, Cid-Ruiz, Li, Monta\~no, and Zhang \cite{CCLMZ} gave a complete answer to the
positivity question in terms of the dimensions of the images of $X$ under these projections.

\begin{theorem}[Castillo -- Cid-Ruiz -- Li -- Monta\~no -- Zhang \cite{CCLMZ}]
	\label{thm_positivity}
	Let 
	$
	X\subset \PP=\PP_\kk^{m_1}\times_\kk\cdots\times_\kk \PP_\kk^{m_p}
	$
	be a $d$-dimensional irreducible multiprojective variety.
	For every $\bn=(n_1,\ldots,n_p)\in\NN^p$ with $|\bn|=d$, we have
	$$
	\deg_{\PP}^\bn(X)>0 \quad \text{ if and only if } \quad \sum_{i \in I} n_i \;\le\; \dim\left(\Pi_I(X)\right) \text{\; for every \;$I\subseteq[p]$.}
	$$
	Moreover, the function
	$
	r_X:2^{[p]}\longrightarrow\NN
	$,
	$I\longmapsto\dim\left(\Pi_I(X)\right)$
	is the rank function of a polymatroid.
\end{theorem}

Thus the support 
$$
\msupp_\PP(X) \;:=\; \big\lbrace \bn \in \NN^p \;\mid\; |\bn| = \dim(X) \text{ and } \deg_{\PP}^\bn(X) > 0\big\rbrace
$$
of the positive multidegrees of an irreducible multiprojective variety $X \subset \PP$  is not an arbitrary subset of $\NN^p$. 
It is precisely the set of lattice points in a polymatroid base polytope. 

\medskip

The tropical analogue of this picture was initiated by He \cite{He}.
Throughout this paper, a \emph{tropical variety} means a finite, pure-dimensional, weighted balanced
rational polyhedral complex with positive integer weights.
Let
$
\Gamma\subset N_\RR = \RR^{m_1}\times\cdots\times\RR^{m_p}
$
be a tropical variety of dimension $d$, and let
$\underline{\Lambda}:=\Lambda_1,\ldots,\Lambda_p$, where
$\Lambda_i\subset\RR^{m_i}$ is a positive tropical divisor.
For each $I\subseteq[p]=\{1,\ldots,p\}$, let
$
\Pi_I:N_\RR \longrightarrow\prod_{i\in I}\RR^{m_i}
$
be the natural projection.
Following  He \cite{He}, the \emph{tropical multidegree} of $\Gamma$ of type $\bn=(n_1,\ldots,n_p)\in\NN^p$, with $|\bn|=d$, is the stable intersection number
$$
\deg_{\underline{\Lambda}}^\bn(\Gamma)
\; := \;
\deg\left(
\Gamma \st
\left(\Pi_1^{-1}(\Lambda_1)\right)^{n_1}
\st\cdots\st
\left(\Pi_p^{-1}(\Lambda_p)\right)^{n_p}
\right).
$$
Since the resulting stable intersection multiplicities
are nonnegative, all tropical multidegrees considered here are nonnegative integers.
For \emph{translation-admissible tropical varieties}, He \cite{He} proved that the positivity of these numbers is
again controlled by the dimensions of the natural projections of $\Gamma$.
This result can be seen, in fact, as a generalization of \autoref{thm_positivity} because the tropicalization of an irreducible variety is translation-admissible (see \cite[Lemma 2.13, Corollary 3.5]{He}). 

\medskip

For an arbitrary tropical variety, however, positivity is not in general governed by the dimensions of the projections of the whole variety.
The correct criterion is \emph{facet-wise}: the projection inequalities need only be witnessed on a single facet; see \autoref{thm_main_crit}.
This naturally leads us to ask when the facet-wise criterion can instead be expressed in terms of the projections of the whole tropical variety.
We introduce two conditions that play distinct roles.
The first, \emph{projection-purity}, requires every (set-theoretic) projection $\Pi_I(\Gamma)$ of $\Gamma$ to be the support of a tropical variety.
The second, \emph{facet-selectability}, requires every vector satisfying the projection inequalities for $\Gamma$ to be witnessed by one facet.
See \autoref{def_two_conditions} for the precise definition of these two conditions.
Under these two hypotheses, the facet-wise criterion globalizes as desired.

\begin{headthm}[\autoref{thm_crit_two_cond}]
\label{thmB}
Let $\Gamma\subset N_\RR = \RR^{m_1}\times\cdots\times \RR^{m_p}$ be a $d$-dimensional projection-pure and facet-selectable tropical variety.
Let $\underline{\Lambda}=\Lambda_1,\ldots,\Lambda_p$ where each $\Lambda_i\subset \RR^{m_i}$ is a positive tropical divisor.
For every $\bn=(n_1,\ldots,n_p)\in\NN^p$ with $|\bn|=d$, we have
$$
\deg_{\underline{\Lambda}}^\bn(\Gamma)>0 \quad \text{ if and only if } \quad \sum_{i \in I} n_i \;\le\; \dim\left(\Pi_I(\Gamma)\right) \text{\; for every \;$I\subseteq[p]$.}
$$
Moreover, the function
$
r_\Gamma:2^{[p]}\longrightarrow\NN
$,
$I\longmapsto\dim\left(\Pi_I(\Gamma)\right)$
is the rank function of a polymatroid.
\end{headthm}

\autoref{thmB} may be viewed as a tropical extension of the positivity criterion in \autoref{thm_positivity}, and it generalizes the criterion of He \cite[Theorem~1.2]{He}. 
As observed in \autoref{rem_translation_two_conditions}, every translation-admissible tropical variety is projection-pure and facet-selectable.
The converse, however, does not hold: \autoref{ex_two_conditions_not_translation} gives a projection-pure and facet-selectable tropical variety that is not translation-admissible. 

\medskip

We denote the \emph{support of positive tropical multidegrees} by 
$$
\TMsupp_{\underline{\Lambda}}(\Gamma)  \;:=\; \Big\lbrace \bn \in \NN^p \;\mid\; |\bn| = \dim(\Gamma) \text{ and } \deg_{\underline{\Lambda}}^\bn(\Gamma) > 0 \Big\rbrace.
$$
Thus the conclusions of \autoref{thmB} say that  ${\rm TMSupp}_{\underline{\Lambda}}(\Gamma)$ is the set of lattice points in a polymatroid base polytope.

\medskip

Classical multidegrees exhibit a stronger feature, their log-concavity, which goes beyond the structure of their positive support.
Writing $\underline H=H_1,\ldots,H_p$, we collect the multidegrees of a $d$-dimensional irreducible multiprojective variety $X\subset\PP$ in its \emph{volume polynomial}
$$
\operatorname{vol}_{X,\underline H}(\ttt)
\;:=\;
\int_\PP \left(t_1H_1+\cdots+t_pH_p\right)^d \,\cdot\, \left[X\right] 
\;=\;
\sum_{|\bn|=d}
\frac{d!}{n_1!\cdots n_p!}\,
\deg_\PP^\bn(X)t_1^{n_1}\cdots t_p^{n_p}.
$$
Since the classes $H_1,\ldots,H_p$ are nef, a fundamental theorem of Br\"and\'en and Huh \cite[Theorem~4.6]{BH} implies that this polynomial is \emph{Lorentzian}. 
Therefore, in the classical setting, the geometry controls not only the support of the multidegrees, but also the  Hodge-theoretic relations among their values required for the Lorentzian property.
For more details on volume polynomials, see the recent survey of Huh \cite{HuhVolume}.

\medskip

The tropical multidegrees of a $d$-dimensional tropical variety $\Gamma \subset N_\RR = \RR^{m_1} \times \cdots \times \RR^{m_p}$ can likewise be assembled into the \emph{tropical volume polynomial}
$$
{\rm tvol}_{\Gamma,\underline\Lambda}(\ttt)
\;:=\;
\sum_{|\bn|=d}
\frac{d!}{n_1!\cdots n_p!}\,
\deg_{\underline\Lambda}^{\bn}(\Gamma) \, t_1^{n_1}\cdots t_p^{n_p}.
$$
\autoref{thmB} recovers one of the two features above: for a projection-pure and facet-selectable tropical variety, the support of this polynomial is M-convex (equivalently, the set of lattice points in a polymatroid base polytope).
 It does not, however, provide the Hodge-theoretic relations on the coefficients required for Lorentzianity. 
Such a failure is, in fact, expected: Babaee and Huh \cite[\S5]{BabaeeHuh} famously constructed a tropical surface in $\mathbb{R}^4$ whose intersection form does not have the signature prescribed by the Hodge index theorem.

We give two counterexamples showing that the conditions we consider in this paper do not force the Lorentzianity of tropical volume polynomials:
\begin{itemize}[\qquad $\bullet$]
	\item A projection-pure and facet-selectable tropical variety that is not translation-admissible and whose tropical volume polynomial is not Lorentzian (see \autoref{prop_non_lorentzian_two_axioms}).
	\item A translation-admissible tropical variety whose tropical
	volume polynomial is not Lorentzian (see \autoref{prop_non_lorentzian_translation}).
\end{itemize}
In both examples, an associated quadratic form has two positive eigenvalues.

\medskip

This contrast leads to a general question:

\begin{question}
	For which tropical varieties
	$
	\Gamma\subset N_\RR=\RR^{m_1}\times\cdots\times\RR^{m_p}
	$
	and which positive tropical divisors
	$\underline\Lambda=\Lambda_1,\ldots,\Lambda_p$,
	$\Lambda_i\subset\RR^{m_i}$,
	is the tropical volume polynomial
	$
	{\rm tvol}_{\Gamma,\underline\Lambda}(\ttt)
	$
	Lorentzian?
\end{question}

Augmented Bergman fans of polymatroids provide a natural positive answer. 
A polymatroid $\sP$ on $[p]=\{1,\ldots,p\}$, together with a cage $\bm=(m_1,\ldots,m_p)$, determines a balanced fan $\Sigma_\sP$ in a product of real vector spaces, with the $i$-th block having dimension $m_i$.
The fan $\Sigma_\sP$ is called the \emph{augmented Bergman fan} of $\sP$.
Eur and Larson \cite{ELP} defined this fan and developed its intersection theory for arbitrary polymatroids.
Previously, the augmented Bergman fan of a matroid was introduced by Braden, Huh, Matherne, Proudfoot, and Wang
\cite{BHMPW}, while Crowley, Huh, Larson, Simpson, and Wang introduced the Bergman fan
of a polymatroid \cite{CHLSW}.
Our main result regarding the augmented Bergman fan of a polymatroid is the following.

\begin{headthm}[\autoref{thm_augmented_main}]
	\label{thmC}
	Let $\sP$ be a polymatroid on $[p]$ of rank $r$ with cage
	$\bm=(m_1,\ldots,m_p)\in\ZZ_+^p$. Choose pairwise disjoint sets
	$E_1,\ldots,E_p$ with $|E_i|=m_i$. Let
	$
	\Sigma_\sP
	\subset \RR^{E_1} \times \cdots \times \RR^{E_p}
	$
	be the augmented Bergman fan of $\sP$.
	Let $\underline{\Lambda}=\Lambda_1,\ldots,\Lambda_p$ where each
	$\Lambda_i\subset \RR^{E_i}$ is a positive tropical divisor.
	Then:
	\begin{enumerate}[\rm (i)]
		\item $\Sigma_\sP$ is projection-pure and facet-selectable.
		\item For every $\bn=(n_1,\ldots,n_p)\in\NN^p$ with $|\bn|=r$, we have
		$$
		\deg_{\underline{\Lambda}}^\bn(\Sigma_\sP)>0 \quad \text{ if and only if } \quad \bn \in B(\sP).
		$$
		\item The tropical volume polynomial
		$$
		{\rm tvol}_{\Sigma_\sP, \underline{\Lambda}}(\ttt) \;=\; \sum_{|\bn| = r} \frac{r!}{n_1!\cdots n_p!}  \, \deg_{\underline{\Lambda}}^\bn(\Sigma_\sP) \, t_1^{n_1} \cdots t_p^{n_p}
		$$
		is Lorentzian.
	\end{enumerate}
\end{headthm}

A couple of words regarding \autoref{thmC} are in order. 
Part (i) shows that augmented Bergman fans satisfy the two properties that we introduce, and then \autoref{thmB} yields the result of part (ii).
On the other hand, part (iii) is a somewhat formal consequence of the fact that the augmented Chow ring of a polymatroid satisfies the \emph{K\"ahler package}.
Versions of the K\"ahler package have been established for Chow rings of matroids \cite{AHK} and their augmented variants \cite{BHMPW}, for Chow rings of polymatroids \cite{PP,CHLSW} and their augmented counterparts \cite{ELP}, and for broader classes of tropical fans (see also \cite{ADH,APfans, BES,AHL,BHMPWsingular}).
Relatedly, Ross~\cite{Ross} introduced and studied the notion of a Lorentzian fan.

\medskip

As a special application of \autoref{thmC}, we recover, in a tropical form, a result of Eur and Larson \cite[Corollary 1.4]{ELP} stating that the (natural) volume polynomial of the augmented Chow ring of a polymatroid equals the exponential generating function of the polymatroid. 
More precisely, in \autoref{cor_augmented_standard_hyperplanes}, we obtain the equality
$$
		{\rm tvol}_{\Sigma_\sP,\underline{\mathcal H}}(\ttt)
\;=\;
r!\sum_{\bn\in B(\sP)\cap\NN^p}\frac{t_1^{n_1}\cdots t_p^{n_p}}{n_1!\cdots n_p!}
$$
for an arbitrary rank-$r$ polymatroid $\sP$ on $[p]$, where $\underline{\mathcal{H}} = \mathcal{H}_1, \ldots, \mathcal{H}_p$ and each $\mathcal{H}_i \subset \RR^{E_i}$ is the standard tropical hyperplane. 

\medskip

By combining \autoref{thmB} and \autoref{thmC}, we obtain the following realizability statement for \emph{arbitrary} polymatroids. 

\begin{headcor}
	Let $N_\RR =  \RR^{m_1} \times \cdots \times \RR^{m_p}$, $\bm = (m_1, \ldots, m_p) \in \ZZ_+^p$, and $\underline{\Lambda} = \Lambda_1, \ldots, \Lambda_p$ where each $\Lambda_i \subset \RR^{m_i}$ is a positive tropical divisor.
	We have the following equality of sets
	\begin{align*}
		\Big\lbrace &\TMsupp_{\underline{\Lambda}}(\Gamma) \;\mid\; \Gamma \subset N_\RR  \text{  is a projection-pure and facet-selectable tropical variety} \Big\rbrace \\
		&\;=\; \Big\lbrace B(\sP) \cap \NN^p \;\mid\; \text{$\sP$ is a polymatroid on $[p]$ with cage $\bm$} \Big\rbrace.
	\end{align*}
	Indeed, the inclusion $\subseteq$ follows from \autoref{thmB} whereas \autoref{thmC} yields the inclusion $\supseteq$.
\end{headcor}

Although \autoref{thmC} places every augmented Bergman fan within the scope of \autoref{thmB}, it does not settle whether augmented Bergman fans are translation-admissible.
We are therefore led to the following natural question:

\begin{question}
	\label{ques_augmented_translation_admissible}
	Let $\Sigma_\sP\subset\prod_{i=1}^p\RR^{E_i}$ be the augmented Bergman
	fan associated with a polymatroid $\sP$. 
	Is the augmented Bergman fan $\Sigma_\sP$ a translation-admissible tropical variety? 
	Equivalently, is $\Sigma_\sP+V$
	pure-dimensional for every rational linear subspace
	$V\subseteq\prod_{i=1}^p\RR^{E_i}$?
\end{question}

\noindent
\textbf{Convention.}
Throughout this paper, unless otherwise stated, all algebraic varieties are defined over the field $\CC$ of complex numbers.
We denote by $\NN = \{0,1,2,\ldots\}$ the set of nonnegative integers and by $\ZZ_+=\{1,2,\ldots\}$ the set of positive integers.

\medskip

\noindent
\textbf{Outline.}
The paper is organized as follows. \autoref{sect_tropical} recalls the notions of tropical multidegrees and bounded rational equivalence and introduces projection-purity and facet-selectability. 
\autoref{sect_toric} proves \autoref{thmB} using toric intersection theory and Minkowski weights.
 \autoref{sec_non_lorentzian} gives \autoref{prop_non_lorentzian_two_axioms} and \autoref{prop_non_lorentzian_translation}, two counterexamples to the Lorentzianity of tropical volume polynomials under the conditions considered in this paper.
 \autoref{sec_augmented} then proves \autoref{thmC}: augmented Bergman fans supply the additional Hodge-theoretic structure that is absent in the counterexamples. 

\section{Tropical multidegrees}
\label{sect_tropical}

In this section, we recall some basic results on tropical varieties and fix the notation used throughout the paper.
Let $p \ge 1$ and $m_1, \ldots, m_p \in \ZZ_+$ be positive integers.
Consider the lattice $N := \ZZ^{m_1+\cdots+m_p}$ and the real space $\NR := N \otimes_\ZZ \RR \cong  \RR^{m_1}\times \cdots \times \RR^{m_p}$.
	Let $m := m_1+\cdots+m_p$.
For every subset $I \subseteq [p] := \{1,\dots,p\}$, write
$$
\Pi_I \;:\; \NR \;\longrightarrow\; \prod_{i\in I}\RR^{m_i}
$$
for the natural projection.
Given a vector of nonnegative integers $\bn=(n_1,\dots,n_p)\in \NN^p$, set
$$
|\bn| \;:=\; n_1+\cdots+n_p
\qquad \text{ and } \qquad
|\bn|_I \;:=\; \sum_{i\in I} n_i \text{ \;for all\; $I \subseteq [p]$.}
$$

Throughout this paper, by the term \emph{tropical variety}, we mean a rational polyhedral complex of pure dimension with positive integer weights that is balanced.
More precisely, we use the following definition:

\begin{definition}
	\label{def_trop_var}
	A \emph{tropical variety} of dimension $d$ in $\NR$ is a weighted rational polyhedral complex
	$(\Gamma,\omega_\Gamma)$ with the following properties:
	\begin{enumerate}[\rm (i)]
		\item $\Gamma$ is a finite rational polyhedral complex in $\NR$.
		\item $\Gamma$ is pure of dimension $d$.
		\item $\omega_\Gamma:\Gamma_d\to \ZZ_{+}$ is a weight function on the set $\Gamma_d$ of facets of $\Gamma$.
		\item $(\Gamma,\omega_\Gamma)$ is balanced: for every codimension-one face $\tau\in \Gamma_{d-1}$, we have
		$$
		\sum_{\sigma\in \Gamma_d, \, \tau\subset \sigma} \omega_\Gamma(\sigma)\,\bv_{\sigma/\tau}=0
		\qquad\text{in } N_\RR/(N_\tau)_\RR,
		$$
	\end{enumerate}
	where $N_\tau$ and $(N_\tau)_\RR$ are the lattice and the linear space parallel to $\tau$, respectively, and $\bv_{\sigma/\tau}$ is the primitive integer generator of $\sigma$ modulo $\tau$.
	The support of the tropical variety is the set
	$
	|\Gamma|:=\bigcup_{\sigma\in \Gamma}\sigma.
	$
	When no confusion is likely, we identify the tropical variety $(\Gamma, \omega_\Gamma)$ with its polyhedral complex $\Gamma$.
\end{definition}

\begin{remark}
	\label{rem_tropical_cycles}
	Following the notation of \cite{AR, AHR}, a \emph{tropical cycle} is an equivalence class of tropical polyhedral complexes, defined up to refinement.
	By standard abuse of notation, when talking about tropical varieties (as defined in \autoref{def_trop_var}), we often allow ourselves to work up to refinements.
	Thus we utilize the flexibility of the notion of tropical cycles.
\end{remark}

To define the notion of \emph{tropical multidegrees} we need the following standard operation on tropical varieties.

\begin{definition}
	Let $\Gamma$ and $\Gamma'$ be tropical varieties in $\NR$ of codimensions
	$c$ and $c'$, respectively.
	First, if $\Gamma$ and $\Gamma'$ \emph{intersect transversely}, then $\Gamma\cap\Gamma'$ either is empty or is a tropical variety of codimension $c+c'$ such that a facet $\tau=\sigma\cap\sigma'$ has weight
	$$
	\omega_{\Gamma\cap\Gamma'}(\tau)
	\;:=\;
	\omega_\Gamma(\sigma)\,\omega_{\Gamma'}(\sigma')\,\big[N:N_\sigma+N_{\sigma'}\big],
	$$
	where $\sigma \in \Gamma$ and $\sigma' \in \Gamma'$ are facets.
	In general, the \emph{stable intersection} of
	$\Gamma$ and $\Gamma'$ can be defined via the limit
	$$
	\Gamma \,\st\, \Gamma' \;:=\; \lim_{\varepsilon\to 0}\, \Gamma\cap (\Gamma'+\varepsilon \bv),
	$$
	where $\bv \in N_\RR$ is a generic vector and the intersections on the right are equipped with the transverse weights above.
	Then $\Gamma \st \Gamma'$ either is empty or it has precisely codimension $c+c'$.
	The idea behind the definition is that, for a generic $\bv \in N_\RR$, the polyhedral complexes $\Gamma$ and $\Gamma' + \varepsilon \bv$ intersect transversely, and so it makes sense to take the limit of these well-behaved intersections.
	For more details on stable intersections, see \cite{JY}, \cite[\S 3.6]{MS}, \cite{MRBook}, \cite{Katz}.
\end{definition}

 If $\Gamma$ is a tropical variety of dimension $0$, we call the sum of all the weights of the points in $\Gamma$ the \emph{degree} of $\Gamma$ and denote it by $\deg(\Gamma)$.
We say that a tropical variety of codimension $1$ is a \emph{tropical divisor} and that one of dimension $1$ is a \emph{tropical curve}.
Following the work of He \cite{He}, we have the following definition of tropical multidegrees.

\begin{definition}[Tropical multidegrees]
	Let $\Gamma\subset \NR$ be a $d$-dimensional tropical variety.
	For each $1 \le i \le p$, let $\Lambda_i \subset \RR^{m_i}$ be a \emph{positive tropical divisor}; this means that $\deg(\Lambda_i \st C) > 0$ for any tropical curve $C$ in $\RR^{m_i}$.
	Let $\underline{\Lambda} = \Lambda_1, \ldots, \Lambda_p$.
	For any $\bn=(n_1,\ldots,n_p) \in \NN^p$ with $|\bn| = d$, we say that the \emph{tropical multidegree} of type $\bn$ with respect to $\underline{\Lambda}$ is given by
	$$
	\deg_{\underline{\Lambda}}^\bn(\Gamma) \;:=\; \deg\left(\Gamma \,\st\, \left(\Pi_1^{-1}(\Lambda_1)\right)^{n_1} \,\st\, \cdots \,\st\, \left(\Pi_p^{-1}(\Lambda_p)\right)^{n_p}\right),
	$$
	where $\left(\Pi_i^{-1}(\Lambda_i)\right)^{n_i}$ denotes the self-stable intersection of $\Pi_i^{-1}(\Lambda_i)$ with itself $n_i$ times.
\end{definition}

\begin{remark}
	\label{rem_std_trop_hyper}
	In practice, a good choice for a positive tropical divisor $\Lambda_i$ in $\RR^{m_i}$ is the  \emph{standard tropical hyperplane} $\mathcal{H}_i$ in $\RR^{m_i}$.
	That is, we can take $\mathcal{H}_i$ to be the set of points $(x_1,\ldots,x_{m_i}) \in \RR^{m_i}$ such that the function $\min\left(x_1,\ldots,x_{m_i}, 0\right)$ achieves the minimum at least twice.
	The standard tropical hyperplane $\mathcal{H}_i \subset \RR^{m_i}$ has weight one on each facet.
\end{remark}

We identify the following two conditions, which will play an important role in our approach.

\begin{definition}
	\label{def_two_conditions}
		Let $\Gamma\subset \NR$ be a $d$-dimensional tropical variety.
		\begin{enumerate}[\rm (i)]
			\item We say that $\Gamma$ is \emph{projection-pure} if, for every $I\subseteq [p]$, the set
			$
			\Pi_I(\Gamma)
			$
			is the support of a tropical variety.

			\item We say that $\Gamma$ is \emph{facet-selectable} if for every $\bn = (n_1,\ldots,n_p) \in \NN^p$ with $|\bn| = d$, the inequalities
			$$
			|\bn|_I \;\le \; \dim\left(\Pi_I(\Gamma)\right) \quad \text{ for all $I \subseteq [p]$}
			$$
			imply that there exists a facet $\sigma$ of $\Gamma$ such that
			$$
			|\bn|_I \;\le \; \dim\left(\Pi_I(\sigma)\right) \quad \text{ for all $I \subseteq [p]$.}
			$$
		\end{enumerate}
	\end{definition}

		\begin{remark}
			These properties play different roles.
			Projection-purity yields submodularity of the projection-dimension function (see
			\autoref{prop_rank_function}).
			Facet-selectability is used to apply the facet-wise criterion of \autoref{thm_main_crit}.
		\end{remark}

		\begin{definition}[{\cite[Definition~2.8]{He}}]
			\label{def_translation_admissible}
			A tropical variety $\Gamma\subset N_\RR$ is \emph{translation-admissible} if, for every
			rational linear subspace $V\subseteq N_\RR$, the Minkowski sum
			$$
			\Gamma+V \;:=\; \big\{\bu+\bv\mid \bu\in\Gamma,\ \bv\in V\big\}
			$$
			is the support of a tropical variety.
			Equivalently, $\Gamma+V$ is pure-dimensional for every
			rational linear subspace $V\subseteq N_\RR$. Indeed, $\Gamma+V$ is the image of the tropical variety $\Gamma\times V$ under the addition map $(\bu,\bv)\mapsto\bu+\bv$; if the image is pure-dimensional, then the standard pushforward weights make it balanced (see \cite[Lemma~2.2]{JY}, \cite[Lemma~3.6.3]{MS}).
		\end{definition}

		\begin{remark}
			\label{rem_translation_two_conditions}
			Every translation-admissible tropical variety is projection-pure and facet-selectable. 
			Indeed, each natural projection of a translation-admissible tropical variety is again a translation-admissible tropical variety by \cite[Lemma~2.11]{He}, which gives projection-purity.
			Facet-selectability is precisely the content of \cite[Lemma~3.1]{He}.
			Thus the two conditions in \autoref{def_two_conditions} are strictly weaker than translation-admissibility, as the following example shows.
		\end{remark}

\begin{example}
\label{ex_two_conditions_not_translation}
Let $p = 3$,  $(m_1,m_2,m_3)=(1,1,2)$ and $N_\RR=\RR\times\RR\times\RR^2$. 
Consider the rational planes
$$
L_1\;:=\;\Span_\RR(\bv_1,\bv_2),
\qquad
\bv_1\;=\;(1,0,1,0),\quad \bv_2\;=\;(0,1,0,1),
$$
and
$$
L_2\;:=\;\Span_\RR(\bu_1,\bu_2),
\qquad
\bu_1\;=\;(2,0,1,0),\quad \bu_2\;=\;(0,2,0,1).
$$
Let $\Gamma$ be the union of the two planes $L_1$ and $L_2$, both with weight one (note that $L_1 \cap L_2 = 0$).
For every
$I\subseteq[3]$, we have
$$
\dim\left(\Pi_I(L_1)\right)
\;=\;
\dim\left(\Pi_I(L_2)\right)
\;=\;
\min\Big\{2,\; \sum_{i\in I}m_i\Big\}.
$$
It follows that $\Gamma$ is facet-selectable. 
For each $I \subseteq [3]$, the rational linear spaces $\Pi_I(L_1)$ and $\Pi_I(L_2)$ have the same dimension. 
Thus $\Gamma$ is projection-pure.
On the other hand, set $V:=\RR\bv_1\subset L_1$. We have
$
\Gamma+V=L_1\cup(L_2+V),
$
where $L_1$ has dimension two, $L_2+V$ has dimension three, and the two spaces meet precisely along
$V$. 
Thus $\Gamma+V$ is not pure-dimensional, and so $\Gamma$ is not translation-admissible.
\end{example}

We now recall the definition of polymatroids.

\begin{definition}
	\label{def_polymatroid_cage}
	A \emph{polymatroid $\sP$ on $[p]=\{1,\ldots,p\}$ with cage $\bm=(m_1,\ldots,m_p)\in\ZZ_+^p$} is a function
	$$
		\rk_\sP \;:\; 2^{[p]} \;\longrightarrow\; \NN
	$$
	satisfying:
	\begin{enumerate}[\rm (i)]
		\item {\rm (Normalization)\;} $\rk_\sP(\varnothing)=0$;
		\item {\rm (Monotonicity)\;} $\rk_\sP(I)\leq\rk_\sP(J)$ whenever $I\subseteq J\subseteq[p]$;
		\item {\rm (Submodularity)\;} $\rk_\sP(I\cap J)+\rk_\sP(I\cup J)\leq\rk_\sP(I)+\rk_\sP(J)$ for all
		$I,J\subseteq[p]$;
		\item {\rm (Cage)\;} $\rk_\sP(i)\leq m_i$ for every $i\in[p]$.
	\end{enumerate}
	The rank of $\sP$ is
	$\rk(\sP):=\rk_\sP([p])$.
	Its \emph{base polytope} is
	$$
	B(\sP) \;:=\;
	\left\{
	\bv = (v_1,\ldots,v_p)\in\RR_{\geq0}^p
	\ \middle|\ 
	\sum_{i\in I}v_i\leq\rk_\sP(I)\ \text{for every }I\subseteq[p],
	\quad \sum_{i=1}^p v_i=\rk(\sP)
	\right\}.
	$$
	A polymatroid with cage $\bm = (1,\ldots,1)$ is called a \emph{matroid}.
	\end{definition}

The following proposition motivates the introduction of the projection-purity condition.

\begin{proposition}
	\label{prop_rank_function}
	Let $\Gamma\subset \NR$ be a tropical variety.
	Assume that $\Gamma$ is projection-pure.
	Then the function
	$$
	r_\Gamma \;: \;2^{[p]} \;\rightarrow\; \NN, \qquad r_\Gamma(I) \;:=\; \dim\left(\Pi_I(\Gamma)\right)
	$$
	is the rank function of a polymatroid on $[p]$ with cage
	$\bm=(m_1,\ldots,m_p)$.
\end{proposition}
\begin{proof}
	Normalization, monotonicity, and the cage inequalities $r_\Gamma(i)\leq m_i$ are clear, so we
	concentrate on submodularity.

	Fix $I,J\subseteq [p]$, and set $K:=I\cup J$.
	By assumption, the projection
	$
	\Gamma_K:=\Pi_K(\Gamma)
	$
	is the support of a tropical variety.
	Its dimension is
	$
	\dim\left(\Gamma_K\right)=r_\Gamma(K).
	$
	To simplify notation, we may assume that $K=[p]$ after replacing $\Gamma$ by $\Gamma_K$.

	Let $\sigma$ be a facet of $\Gamma$, and set $L:= {\rm Par}(\sigma) = \Span_\RR\big\lbrace\fu - \fv \mid \fu, \fv \in \sigma\big\rbrace$ to be the linear space parallel to $\sigma$.
	Then
	$$
	\dim(L) \;=\; \dim(\sigma) \;=\; r_\Gamma([p]).
	$$
	In the product $N_\RR = \prod_{i\in [p]}\RR^{m_i}$, the kernels of $\Pi_I$ and $\Pi_J$ intersect trivially because $I\cup J=[p]$.
	Therefore
	$$
	\left(L\cap \Ker(\Pi_I)\right) \,\oplus\, \left(L\cap \Ker(\Pi_J)\right) \;\subseteq\; L \,\cap\, \Ker(\Pi_{I\cap J}).
	$$
	Taking dimensions and using rank-nullity on $L$ gives
	$$
	\dim\left(\Pi_{I\cap J}(L)\right) + r_\Gamma([p]) \;\le\; \dim \left(\Pi_I(L)\right) + \dim\left(\Pi_J(L)\right).
	$$
	Since $L$ is the linear space parallel to $\sigma$, we have
	$
	\dim\left(\Pi_S(\sigma)\right) = \dim\left(\Pi_S(L)\right)
	$
	for all $S\subseteq [p]$.
	Thus, every facet $\sigma$ of $\Gamma$ satisfies
	$$
	\dim\left(\Pi_{I\cap J}(\sigma)\right) + r_\Gamma([p]) \;\le\; \dim \left(\Pi_I(\sigma)\right)+\dim\left(\Pi_J(\sigma)\right).
	$$
	Now fix a facet $\tau$ of $\Gamma$ so that $r_\Gamma(I \cap J) = \dim\left(\Pi_{I \cap J}(\tau)\right)$.
	Then we get the inequalities
	\begin{align*}
		r_\Gamma(I\cap J) + r_\Gamma([p]) &\;=\; \dim\left(\Pi_{I \cap J}(\tau)\right) + r_\Gamma([p]) \\
		&\;\le\; \dim \left(\Pi_I(\tau)\right)+\dim\left(\Pi_J(\tau)\right)\\
		&\;\le\; \max_{\text{$\sigma$ facet of $\Gamma$}}\Big\lbrace\dim\left( \Pi_I(\sigma)\right)+\dim\left( \Pi_J(\sigma)\right)\Big\rbrace \\
		&\;\le\;
		r_\Gamma(I)+r_\Gamma(J).
	\end{align*}
	This is the desired submodular inequality.
\end{proof}

To study tropical multidegrees, we need the notion of \emph{bounded rational equivalence} as developed by Allermann, Hampe and Rau \cite{AR, AHR}.
This provides the link to the toric methods and Minkowski weights used in \autoref{sect_toric}.
If two tropical cycles $\Gamma \subset N_\RR$ and  $\Gamma' \subset N_\RR$ are bounded rationally equivalent, we write $\Gamma \bdeq \Gamma'$.

\begin{definition}
	Let $\sigma\subset N_\RR$ be a polyhedron. Its \emph{recession cone} is
	$$
		\Rec(\sigma)
		\;:=\;
		\bigl\{\bv\in N_\RR\mid \bu+\RR_{\geq0}\bv\subseteq\sigma
		\text{ for every }\bu\in\sigma\bigr\}.
	$$
	If $\Gamma\subset N_\RR$ is a tropical cycle, then, after refining $\Gamma$ if necessary,
	$$
		\bigl\{\Rec(\sigma)\mid\sigma\text{ is a face of }\Gamma\bigr\}
	$$
	is a fan.
	This fan has a natural weighting that makes it a tropical fan cycle, called the
	\emph{recession fan} of $\Gamma$ and denoted $\Rec(\Gamma)$; see	\cite[Definition~5.1]{AHR}.
\end{definition}

We will need the following properties of bounded rational equivalence.

\begin{theorem}[{Allermann, Hampe and Rau \cite[Proposition 3.3, Theorem 5.3]{AHR}}]
	\label{thm_AHR}
	Let $\Gamma \subset N_\RR$ be a tropical cycle and $\bv \in N_\RR$.
	Then:
	\begin{enumerate}[\rm (i)]
		\item $\Gamma \,\bdeq\, \Gamma + \bv$.
		\item $\Gamma \,\bdeq\, \Rec(\Gamma)$.
	\end{enumerate}
\end{theorem}

Finally, we have the following proposition stating that tropical multidegrees are invariant under bounded rational equivalence.

\begin{proposition}
	\label{prop_multideg_bdeq}
	Let
	$\Gamma,\Gamma'\subset \NR$ be $d$-dimensional tropical cycles with $\Gamma\bdeq \Gamma'$, and for
	each $i$ let $\Lambda_i,\Lambda_i'\subset \RR^{m_i}$ be tropical divisors with $\Lambda_i\bdeq \Lambda_i'$.
	Let $\underline{\Lambda} = \Lambda_1,\ldots,\Lambda_p$ and $\underline{\Lambda'} = \Lambda_1',\ldots,\Lambda_p'$.
	 Then, for every $\bn=(n_1,\dots,n_p)\in \NN^p$ with $|\bn|=d$, we have
	$
	\deg_{\underline{\Lambda}}^\bn(\Gamma)
	=
	\deg_{\underline{\Lambda'}}^\bn(\Gamma').
	$
\end{proposition}
\begin{proof}
	By \cite[Proposition~3.2(i)]{AHR}, bounded rational equivalence is preserved after taking products
	with another tropical cycle. Since
	$
	\Pi_i^{-1}(\Lambda_i)
	=
	\Lambda_i\times \prod_{j\neq i}\RR^{m_j},
	$
	we obtain
	$
	\Pi_i^{-1}(\Lambda_i) \bdeq \Pi_i^{-1}(\Lambda_i')
	$
	for each  $i$.
	Then \cite[Proposition~3.2(iv)]{AHR} shows that stable intersection with a fixed tropical cycle preserves bounded rational equivalence.
	Moreover, \cite[Proposition~3.2(v)]{AHR} shows that the degree of a zero-dimensional cycle is preserved under bounded rational equivalence.
	By applying these statements repeatedly, we obtain the claimed equality.
\end{proof}

\section{Toric methods and Minkowski weights}
	\label{sect_toric}

	In this section, we prove \autoref{thmB} using toric methods.
	We exploit the description of Chow cohomology of complete toric varieties in terms of Minkowski weights due to Fulton and Sturmfels \cite{FS}.
	For the basic results and notation on toric varieties, we follow Fulton \cite{Fu} and Cox, Little, and Schenck \cite{CLS}.
	We briefly recall the relevant definitions.

	We continue using the notation introduced in \autoref{sect_tropical}.
	In particular, we have positive integers $p \ge 1$ and $m_1, \ldots, m_p \in \ZZ_+$, and we set $N := \ZZ^{m_1+\cdots+m_p}$, $\NR := N \otimes_\ZZ \RR \cong  \RR^{m_1}\times \cdots \times \RR^{m_p}$ and $m := m_1+\cdots+m_p$.
	Let
	$
	M:=N^\vee:=\operatorname{Hom}_\ZZ(N,\ZZ)
	$
	be the lattice dual to $N$. 
	Thus $N$ and $M$ are, respectively, the cocharacter and character lattices of the algebraic torus $T_N$.
	Let $\Sigma$ be a complete fan in $N_\RR$, and let $0\le k\le m$.
	A \emph{codimension-$k$ Minkowski weight} on $\Sigma$ is a function on codimension-$k$ cones
	$$
		\omega \;:\; \Sigma^{(k)}\; \longrightarrow\; \ZZ
	$$
	such that, for every codimension-$(k+1)$ cone $\tau\in \Sigma^{(k+1)}$, we have the balancing condition
	$$
	\sum_{\sigma\in \Sigma^{(k)},\, \tau\subset \sigma} \omega(\sigma)\,\bv_{\sigma/\tau}=0
	\qquad\text{in } N_\RR/(N_\tau)_\RR.
	$$
The \emph{group of codimension-$k$ Minkowski weights} on $\Sigma$ is given by
$$
\MW^k(\Sigma) \;:=\; \big\lbrace \text{codimension-$k$ Minkowski weights on $\Sigma$}\big\rbrace \;\subset\; \ZZ^{\Sigma^{(k)}}.
$$
We repeatedly use the correspondence of Fulton--Sturmfels \cite{FS}: if $\Sigma$ is a complete fan, then for every $k$ there is a natural isomorphism between the group of codimension-$k$ Minkowski weights on $\Sigma$ and the operational Chow cohomology group:
$$
	A^k(X_\Sigma) \;\cong\; \MW^k(\Sigma).
$$
We freely regard codimension-$k$ Minkowski weights on $\Sigma$ as classes in $A^k(X_\Sigma)$.
Under this correspondence, the cup product in $A^\bullet(X_\Sigma)$ is transported to the product of Minkowski weights computed by the \emph{fan displacement rule} of Fulton--Sturmfels \cite{FS}.
For complete rational fans, this product is the stable intersection product of the corresponding tropical fan cycles; see \cite{Katz} and \cite[\S 6.7]{MS}.
In particular, in codimension one, we use the natural isomorphisms 
$$
\MW^1\left(\Sigma\right) \;\cong\; A^1\left(X_\Sigma\right) \;\cong\; \Pic\left(X_\Sigma\right)
$$
(see \cite[Corollary 3.4]{FS}).
Therefore, to any tropical fan divisor $\Lambda$ compatible with $\Sigma$ (i.e., a refinement of $\Lambda$ is a subfan of $\Sigma$), we can associate a codimension-one Minkowski weight on $\Sigma$ and a torus-invariant Cartier divisor on $X_\Sigma$.

We begin with the following lemma, which is a direct consequence of \cite[Theorem A]{CCLMZ}.

\begin{lemma}
	\label{lem_pos_multdeg}
	Let $Y$ be an irreducible complete variety of dimension $d$, and let $\LL_1,\dots,\LL_p$ be globally generated line bundles on $Y$.
	For each $i$, let $\varphi_i:Y\rightarrow \PP\big(\HH^0\left(Y, \LL_i\right)^{\scriptscriptstyle\vee}\big)$ be the morphism defined by $\LL_i$, and let
	$$
	\varphi_I \;:\; Y \;\longrightarrow \prod_{i \in I} \PP\big(\HH^0\left(Y, \LL_i\right)^{\scriptscriptstyle\vee}\big) \qquad \text{ for each $I \subseteq [p]$}.
	$$
	Fix $\bn=(n_1,\dots,n_p)\in \NN^p$ with $|\bn|=d$.
	Then
	$$
	\int_Y c_1(\LL_1)^{n_1}\cdots c_1(\LL_p)^{n_p} \smallfrown [Y] \;>\; 0 \qquad \text{if and only if} \qquad  |\bn|_I \;\le\; \dim\left(\varphi_I(Y)\right) \;\; \text{ for all $I \subseteq [p]$}.
	$$
\end{lemma}
	\begin{proof}
		Consider the product of projective spaces $\PP = \PP\big(\HH^0\left(Y, \LL_1\right)^{\scriptscriptstyle\vee}\big) \times \cdots \times \PP\big(\HH^0\left(Y, \LL_p\right)^{\scriptscriptstyle\vee}\big)$ and the morphism  $\varphi = \varphi_{[p]} : Y \rightarrow \PP$.
		For each $I \subseteq [p]$, let
		$$
		p_I \;:\; \PP \;\longrightarrow\; \prod_{i \in I} \PP\big(\HH^0\left(Y, \LL_i\right)^{\scriptscriptstyle\vee}\big)
		$$
		be the natural projection.
		For each $i$, notice that $\LL_i \cong \varphi^*\left(\OO_\PP(\ee_i)\right)$ where $\ee_i = (0,\ldots,1,\ldots,0) \in \NN^p$ is the $i$-th elementary basis vector.
		Let $W = \varphi(Y) \subset \PP$.
		By the projection formula, we get the equality
		\begin{align*}
			\int_Y c_1(\LL_1)^{n_1}\cdots c_1(\LL_p)^{n_p} \smallfrown [Y] &\;=\;  \deg\left(Y/W\right) \,\int_\PP c_1(\OO_\PP(\ee_1))^{n_1}\cdots c_1(\OO_\PP(\ee_p))^{n_p} \smallfrown [W] \\
			&\;=\; \deg\left(Y/W\right)\,  \deg_\PP^\bn(W).
		\end{align*}
		Recall that, if $Y\rightarrow W$ is not generically finite, then $\deg\left(Y/W\right)=0$.
		Finally, by \cite[Theorem A]{CCLMZ}, we obtain $\deg_\PP^\bn(W)>0$ if and only if $|\bn|_I \le \dim\left(p_I(W)\right) = \dim\left(\varphi_I(Y)\right)$ for all $I \subseteq [p]$.
	\end{proof}
	
The next proposition is one of our main tools. 
Its proof relies crucially on the Fulton–Sturmfels formula \cite[Lemma 4.4]{FS} for the Chow class of a subtorus compactification.

\begin{proposition}
	\label{prop_recover_pullback_fan}
	Let $\Sigma$ be a complete fan in $N_\RR$, let $L\subset N_\RR$ be a rational linear
	subspace of dimension $d$, and set $N_L:=N\cap L$.
	Let $\TL\subset T_N$ be the subtorus with
	cocharacter lattice $N_L$, and let
	$
	Y \;:=\; \overline{\TL} \;\subset\; X_\Sigma
	$
	be the closure of the torus $T_L$ in $X_\Sigma$.
	Then the following statements hold:
	\begin{enumerate}[\rm (i)]
		\item There exists a smooth complete refinement $\widehat{\Sigma}$ of $\Sigma$ such that
		$
		\widehat{\Sigma}_L := \big\lbrace \sigma \cap L \mid \sigma \in \widehat{\Sigma} \big\rbrace
		$
		is a subfan of $\widehat{\Sigma}$.
		Let
		$\widehat{Y} \;:=\; \overline{\TL} \;\subset\; X_{\widehat{\Sigma}}
		$
		be the closure of the torus $T_L$ in $X_{\widehat{\Sigma}}$.

		\item The natural toric morphism $j:X_{\widehat{\Sigma}_L}\rightarrow X_{\widehat{\Sigma}}$ is a closed immersion.
		In particular, we have $\widehat{Y}\cong X_{\widehat{\Sigma}_L}$.
		\item The function
		$$
			\omega_{L} \;:\; \widehat{\Sigma}^{(m-d)} \;\longrightarrow\; \ZZ,
			\qquad
			\omega_{L}(\sigma) \;:=\;
			\begin{cases}
				1 & \text{if } \sigma\in\widehat{\Sigma}_L\\
				0 & \text{otherwise } 
			\end{cases}
		$$
		is a codimension-$(m-d)$ Minkowski weight on $\widehat{\Sigma}$.
		\item The fundamental class $\big[\widehat{Y}\big] \in A^{m-d}\big(X_{\widehat{\Sigma}}\big)$ of $\widehat{Y}$ equals the class determined by $\omega_{L}$.
	\end{enumerate}
\end{proposition}
\begin{proof}
	Let $M_L := \Hom_\ZZ\left(N_L, \ZZ\right)$ be the lattice dual to $N_L$.
	Since $N_L \subset N$ is a saturated sublattice, we have a natural surjection $M \surjects M_L$.
	
	\smallskip
	
	(i) Let $\Sigma'$ be a complete rational fan refining $\Sigma$ such that $L$ is the support of a subfan of $\Sigma'$.
	Then the result follows by applying \cite[Theorem~11.1.9]{CLS} to the complete fan $\Sigma'$.
	
	\smallskip

	(ii) 
	As in \cite[Proposition~11.2.14 and Theorem~3.A.5]{CLS}, it suffices to show that the natural toric morphism $j : X_{\widehat\Sigma_L}\to X_{\widehat\Sigma}$ is a closed immersion. 
	We can check this locally on $X_{\widehat\Sigma}$.
	Choose $\sigma \in \widehat{\Sigma}$ and $U_\sigma = \Spec\left(\CC[\sigma^\vee \cap M]\right) \subset X_{\widehat{\Sigma}}$.
	We have that $j^{-1}(U_\sigma) = U_\delta \subset X_{\widehat{\Sigma}_L}$ where $\delta = \sigma \cap L$ and $U_\delta = \Spec\left(\CC[\delta^\vee \cap M_L]\right)$.
	The result now follows because we have a surjective map of affine semigroup rings 
	$$
	\CC\big[\sigma^\vee \cap M\big] \;\surjects\; \CC\big[\delta^\vee \cap M_L\big].
	$$
	We now check this claim.
	Let $\bv \in \delta^\vee \cap M_L$ and take a lattice point $\bw \in M$ mapping to $\bv$.
	Since $\delta$ is a face of $\sigma$ and $\delta = \sigma \cap L$, there exists $\bw' \in M \cap L^\perp$ such that $\langle \bw', \bu \rangle > 0$ for all $\bu \in \sigma \setminus \delta$.
	Therefore, for $q \gg 0$, we have that $\bw + q\bw' \in \sigma^\vee$ and that its image under $M \surjects M_L$ is $\bv$.
	
	\smallskip
	
	(iii) 
	We have that $\widehat{\Sigma}_L$ is complete in $L$ because $\widehat{\Sigma}$ is complete in $N_\RR$.
	It then follows that $\omega_{L}$ restricted to $\widehat{\Sigma}_L$ yields a Minkowski weight.
	Then, extending by zero outside the subspace $L \subset N_\RR$, we get that $\omega_L$ defines a Minkowski weight on $\widehat{\Sigma}$.

	\smallskip

	(iv) Since $\widehat{Y}$ is the closure of the subtorus $\TL\subset T_N$ with cocharacter lattice $N_L$, \cite[Lemma~4.4]{FS} applies directly to the torus compactification $\widehat{Y}\subset X_{\widehat{\Sigma}}$.
	Hence, for every generic lattice vector $\bv\in N$, the fundamental class of $\widehat{Y}$ is given by
	$$
	\big[\widehat{Y}\big]
	\;=\;
	\sum_{\tau\in\widehat{\Sigma}(\bv)} \big[N:N_L+N_\tau\big]\,\big[V(\tau)\big]
	\qquad\text{in } A^{m-d}\left(X_{\widehat{\Sigma}}\right),
	$$
	where
	$$
	\widehat{\Sigma}(\bv)
	\;:=\;
	\big\{ \tau\in\widehat{\Sigma} \;\mid\; \text{$L + \bv$ meets $\tau$ in exactly one point} \big\};
	$$
	notice that $\dim(\tau)=m-d \text{ and } (L+\bv)\cap\Relint(\tau)\neq\varnothing$.
	By \cite[Theorem~3.1]{FS}, there is a unique class
	in $A^{m-d}(X_{\widehat{\Sigma}})$ corresponding to the Minkowski weight $\omega_L$.
	To conclude the proof, by \cite[Proposition 4.1]{FS},  it suffices to show the following equality
	\begin{equation}
		\label{eq_class_minkowski_L}
		\int_{X_{\widehat{\Sigma}}} \big[\widehat{Y}\big] \cdot \big[V(\sigma)\big] \;=\; \sum_{\tau\in\widehat{\Sigma}(\bv)} \big[N:N_L+N_\tau\big]\,\int_{X_{\widehat{\Sigma}}} \big[V(\tau)\big] \cdot \big[V(\sigma)\big] \;=\; \begin{cases}
			1 &  \text{if } \sigma \in \widehat{\Sigma}_L\\
			0 &  \text{otherwise }
		\end{cases}
	\end{equation}
	for all $\sigma \in \widehat{\Sigma}^{(m-d)}$.
	The intersection  $V(\tau) \cap V(\sigma)$ equals $V(\tau + \sigma)$ if $\tau$ and $\sigma$ span a  cone in $\widehat{\Sigma}$, and $\varnothing$ otherwise.
	Therefore, the above sum restricts to cones $\tau \in \widehat{\Sigma}(\bv)$ such that $\tau + \sigma$ is also a cone in $\widehat{\Sigma}$.

	Fix $\sigma \in \widehat{\Sigma}^{(m-d)}$.
	We divide the proof of the claimed equality \autoref{eq_class_minkowski_L} into two cases.

	{\sc Case 1:}
	Assume $\sigma \in \widehat{\Sigma}_L$. Since $\sigma$ is a $d$-dimensional cone in $L$, it spans $L$,  and so $N_\sigma = N_L$.
	The genericity of $\bv$ forces $L$ to be transverse to any $\tau \in \widehat{\Sigma}(\bv)$, and so we get
	$$
	\Span_\RR(\tau) + \Span_\RR(\sigma) \;=\; \Span_\RR(\tau) + L \;=\; \RR^m.
	$$
	Therefore $\dim(\tau + \sigma) = m$, and so we seek the $(m-d)$-dimensional cones $\tau \in \widehat{\Sigma}(\bv)$ such that $\gamma = \tau + \sigma$ is a maximal cone in $\widehat{\Sigma}$.
	Since $\widehat{\Sigma}$ is smooth, every maximal cone $\gamma$ containing $\sigma$ uniquely decomposes as $\gamma = \tau + \sigma$, where $\tau$ is an $(m-d)$-dimensional face.

	Consider the natural projection map $\pi: N_{\mathbb{R}} \surjects N_{\mathbb{R}} / L$.
	Then $\Star(\sigma) = \big\lbrace \overline{\gamma} \mid \sigma \subseteq \gamma \in \widehat{\Sigma} \big\rbrace$ is a complete fan in the quotient $N(\sigma)_\RR = N_{\mathbb{R}} / L$.
	The condition $(L+\bv) \cap \Relint(\tau) \neq \emptyset$ holds if and only if $\pi(\bv)$ lies in the relative interior of $\pi(\tau)$.

	Since the quotient fan $\Star(\sigma)$ is complete and $\bv$ is generic, $\pi(\bv)$ lies in the interior of exactly one maximal cone in the quotient.
	This implies there is a \emph{unique} cone $\tau \in \widehat{\Sigma}(\bv)$ such that $\tau + \sigma \in \widehat{\Sigma}$ is a maximal cone.
	For this unique $\tau$, because the maximal cone $\tau + \sigma$ is smooth, its primitive ray generators form a basis for $N$.
	As a consequence, $N_\tau \oplus N_\sigma = N$, and so by substituting $N_L = N_\sigma$, we obtain $N_L + N_\tau = N$. 
	The lattice index is therefore:
	$$
	\big[N : N_L + N_\tau\big] \;=\; \big[N : N\big] = 1.
	$$
	The degree of a point $[V(\tau + \sigma)]$ is exactly $1$.
	Finally, the sum in \autoref{eq_class_minkowski_L} collapses to a single term with coefficient $1$, yielding that the degree of $[\widehat{Y}] \cdot [V(\sigma)]$ is equal to one.

	{\sc Case 2:}
	Assume $\sigma \notin \widehat{\Sigma}_L$.
	For this case, we repeatedly use the Orbit–Cone Correspondence, see \cite[Theorem 3.2.6]{CLS}.
	The subvariety $X_{\widehat{\Sigma}_L}$ is a toric variety with torus $T_L$, stratified by $T_L$-orbits $O_{\gamma, L}$ for $\gamma \in \widehat{\Sigma}_L$.
	The ambient toric variety $X_{\widehat{\Sigma}}$ has torus $T_N$, stratified by $T_N$-orbits $O_{\gamma, N}$ for $\gamma \in \widehat{\Sigma}$.

	Since $\widehat{\Sigma}_L$ is a subfan of $\widehat{\Sigma}$, the closed toric immersion $j: X_{\widehat{\Sigma}_L} \hookrightarrow X_{\widehat{\Sigma}}$ maps each $T_L$-orbit $O_{\gamma, L}$ into the corresponding $T_N$-orbit $O_{\gamma, N}$.
	Thus, the subtorus closure $\widehat{Y} = j(X_{\widehat{\Sigma}_L})$ sits inside the ambient variety $X_{\widehat{\Sigma}}$ as:
	$$
	\widehat{Y} \;=\; \bigsqcup_{\gamma \in \widehat{\Sigma}_L} j\left(O_{\gamma, L}\right) \;\subset\; \bigsqcup_{\gamma \in \widehat{\Sigma}_L} O_{\gamma, N}.
	$$
	This implies that $\widehat{Y}$ only intersects a $T_N$-orbit if the corresponding cone belongs to $\widehat{\Sigma}_L$.

	On the other hand, the test class $V(\sigma) \subset X_{\widehat{\Sigma}}$ is the closure of the $T_N$-orbit $O_{\sigma, N}$.
	This orbit closure is the union of orbits corresponding to cones containing $\sigma$:
	$$
	V(\sigma) \;=\; \bigsqcup_{\substack{\gamma \in \widehat{\Sigma},\, \sigma \subseteq \gamma}} O_{\gamma, N}.
	$$
	Since $\sigma \notin \widehat{\Sigma}_L$, we also have $\gamma \notin \widehat{\Sigma}_L$ for all $\gamma \supseteq \sigma$.
 	Therefore, we have $\widehat{Y} \cap V(\sigma) = \varnothing$.
 	This concludes the proof of the proposition.
\end{proof}

The next lemma translates the notion of positive tropical divisor to the toric setting. 
We show that a Cartier divisor associated to a positive tropical fan divisor is basepoint-free and big.

\begin{lemma}
		\label{lem_refine_div}
	Let $\Lambda_i \subset \RR^{m_i}$ be a positive tropical fan divisor. 
	Then there is a smooth complete fan $\Sigma_i \subset \RR^{m_i}$ and a basepoint-free and big torus-invariant Cartier divisor $D_i$ on $X_{\Sigma_i}$ associated to a refinement of $\Lambda_i$.
\end{lemma}
\begin{proof}
	By using \cite[Theorem~11.1.9]{CLS}, we obtain a smooth complete fan $\Sigma_i \subset \RR^{m_i}$ such that a refinement of $\Lambda_i$ is a subfan of $\Sigma_i$.
		Let $\omega_{\Lambda_i}\in\MW^1(\Sigma_{i})$ be the Minkowski weight represented by $\Lambda_i$.
	Let $D_i$ be a Cartier divisor on $X_{\Sigma_{i}}$ representing the class $\omega_{\Lambda_i} \in \MW^1\left(\Sigma_{i}\right) \cong A^1\big(X_{\Sigma_{i}}\big) \cong {\rm Pic}\big(X_{\Sigma_{i}}\big)$.
	For every wall $\tau\in\Sigma_{i}^{(1)}$, we obtain
	$$
	D_{i}\cdot V(\tau) \;=\; \omega_{\Lambda_i}(\tau) \;\ge\; 0
	$$
	(see \cite[Proposition~4.1]{FS}).
	Then \cite[Theorem~6.3.12]{CLS} implies that $D_{i}$ is basepoint-free on $X_{\Sigma_{i}}$.

	We assume by contradiction that $D_i$ is not big on $X_{\Sigma_i}$.
	Let
	$$
	\varphi_i \;:\; X_{\Sigma_i} \;\longrightarrow\;
	\PP\bigl(\HH^0(X_{\Sigma_i},\OO_{X_{\Sigma_i}}(D_i))^{\scriptscriptstyle\vee}\bigr)
	$$
	be the morphism defined by $D_i$.
	This morphism is toric and restricts to a homomorphism $\psi_i:T_i\to T_i'$ on the dense tori.
	Since $D_i$ is not big, there is some $\bv \in \ZZ^{m_i}$ such that the one-dimensional subtorus $T_\bv := \IM(\lambda^\bv : \CC^* \rightarrow T_i) \subset T_i$ lies inside the kernel of $\psi_i$.
	Consider the corresponding curve $C_\bv := \overline{T_\bv} \subset X_{\Sigma_i}$ and the corresponding line $\ell_\bv := \RR \cdot \bv \subset \RR^{m_i}$.
	Then, by utilizing \autoref{prop_recover_pullback_fan}, we obtain
	$$
	\deg\left(\Lambda_i \st \ell_\bv\right) \;=\; D_i \cdot C_\bv \;=\; 0;
	$$
	a contradiction to the positivity of $\Lambda_i$. 
	Therefore $D_i$ is big, and this completes the proof of the lemma.  
\end{proof}

As a first important case, by combining the previous results, we can characterize the positivity of the tropical multidegrees of a rational linear space in $N_\RR$.

\begin{proposition}
	\label{prop_trop_multdeg_L}
	Let $L \subset N_\RR$ be a $d$-dimensional rational linear space and $\Lambda_1 \subset \RR^{m_1},\ldots,\Lambda_p\subset \RR^{m_p}$ be positive tropical fan divisors.
	Let $\underline{\Lambda} = \Lambda_1 ,\ldots,\Lambda_p$.
	Fix $\bn=(n_1,\dots,n_p)\in \NN^p$ with $|\bn|=d$.
	Then $\deg_{\underline{\Lambda}}^\bn(L) > 0$ if and only if
	$
	|\bn|_I \le \dim\left(\Pi_I(L)\right)
	$ 
	for all $I \subseteq [p]$.
\end{proposition}
\begin{proof}
	By using \autoref{lem_refine_div}, we choose a smooth complete fan $\Sigma_i \subset \RR^{m_i}$ and a basepoint-free and big torus-invariant Cartier divisor $D_i$ on $X_{\Sigma_i}$ such that $D_i$ corresponds to $\Lambda_i$.
	Consider the smooth complete fan $\Sigma = \Sigma_1 \times \cdots \times \Sigma_p$ in $N_\RR$ and the smooth  toric variety $X_\Sigma = X_{\Sigma_1} \times \cdots \times X_{\Sigma_p}$.
	By an abuse of notation, we also denote by $D_i$ the pullback of $D_i$ to $X_\Sigma$ and by $\Pi_I : X_\Sigma \rightarrow \prod_{i\in I} X_{\Sigma_i}$ the natural projection. 

	We now apply \autoref{prop_recover_pullback_fan} to the fan $\Sigma$ and the linear space $L$.
	Let $N_L:=N\cap L$,  $\TL\subset T_N$ be the subtorus with
	cocharacter lattice $N_L$ and $Y := \overline{\TL} \subset X_\Sigma$.
	By \autoref{prop_recover_pullback_fan}, there is a smooth complete refinement $\widehat{\Sigma}$ of $\Sigma$ such that the fundamental class $\big[\widehat{Y}\big] \in A^{m-d}\left(X_{\widehat{\Sigma}}\right)$ of $\widehat{Y}  \subset X_{\widehat{\Sigma}}$ equals the class determined by the Minkowski weight $\omega_{L}$ (this Minkowski weight corresponds to the tropical cycle given by $L$).
	The natural toric morphism $q : X_{\widehat{\Sigma}} \longrightarrow X_{\Sigma}$ satisfies $q_*\big(\big[\widehat{Y}\big]\big) = \big[Y\big]$.
	Therefore, the correspondence of Fulton--Sturmfels \cite{FS} gives the equality
	$$
	\deg_{\underline{\Lambda}}^\bn(L) \;=\; \int_{X_{\widehat{\Sigma}}} \widehat{D}_1^{n_1} \cdots \widehat{D}_p^{n_p} \cdot \big[\widehat{Y}\big],
	$$
	where $\widehat{D}_i$ is the pullback of $D_i$ to $X_{\widehat{\Sigma}}$, and then the projection formula yields
	$$
	\deg_{\underline{\Lambda}}^\bn(L) \;=\; \int_{X_{\widehat{\Sigma}}} \widehat{D}_1^{n_1} \cdots \widehat{D}_p^{n_p} \cdot \big[\widehat{Y}\big] \;=\; \int_{X_\Sigma} D_1^{n_1} \cdots D_p^{n_p} \cdot \big[Y\big].
	$$
	Let
	$
		\varphi_i:X_{\Sigma_i}\longrightarrow
		\PP\bigl(\HH^0(X_{\Sigma_i},\OO_{X_{\Sigma_i}}(D_i))^{\scriptscriptstyle\vee}\bigr)
		=: \PP^{s_i}
	$
	be the morphism defined by $D_i$. 
	We get a morphism
		$$
		\varphi \;:\; X_\Sigma = X_{\Sigma_1} \times \cdots \times X_{\Sigma_p}  \;\longrightarrow\; \PP^{s_1} \times \cdots \times \PP^{s_p}.
		$$
		For each $I \subseteq [p]$, let $\varphi_I : X_\Sigma \longrightarrow \prod_{i \in I} \PP^{s_i}$
		be the natural map.
		Since $D_i$ is big, the map $\varphi_i:X_{\Sigma_i}\longrightarrow \PP^{s_i}$ is generically finite onto its image, and so the corresponding map of dense tori $\psi_i : T_i \subset X_{\Sigma_i} \longrightarrow T_i' \subset \PP^{s_i}$ is finite onto its image.
		This implies that the restriction of $\prod_{i \in I} T_i \longrightarrow \prod_{i \in I} T_i'$ to $\Pi_I\left(T_L\right)$ is also finite onto its image.
		 Since $T_L$ is dense in $Y$, we obtain
		$$
		\dim\left(\varphi_I(Y)\right)
		\;=\; \dim\left(\varphi_I(T_L)\right) \;=\; \dim\left(\Pi_I(T_L)\right) \;=\; \dim\left(\Pi_I(L)\right)
		$$
		for all $I \subseteq [p]$. 
		Finally, the result of the proposition follows from \autoref{lem_pos_multdeg}.
\end{proof}

We are now ready to prove the following positivity criterion for arbitrary tropical varieties. 

\begin{theorem}
	\label{thm_main_crit}
	Let
$
\Gamma\subset N_\RR=\RR^{m_1}\times\cdots\times \RR^{m_p}
$
be a $d$-dimensional tropical variety, and let $\underline{\Lambda}=\Lambda_1,\ldots,\Lambda_p$, where each $\Lambda_i\subset \RR^{m_i}$ is a positive tropical divisor.
Let $\sigma_1,\ldots,\sigma_k$ be the facets of $\Gamma$.
For every $\bn=(n_1,\ldots,n_p)\in\NN^p$ with $|\bn|=d$, we have
$\deg_{\underline{\Lambda}}^\bn(\Gamma)>0$
if and only if there exists a facet $\sigma_j$ such that
$$
\sum_{i \in I} n_i \;\le\; \dim\left(\Pi_I(\sigma_j)\right)
$$
for every $I\subseteq[p]$.
\end{theorem}
\begin{proof}
	Fix $\bn=(n_1,\dots,n_p)\in \NN^p$ with $|\bn|=d$.
	Consider the codimension-$d$ tropical cycle
	$$
	\Delta \;:=\; \left(\Pi_1^{-1}(\Lambda_1)\right)^{n_1} \st \cdots \st \left(\Pi_p^{-1}(\Lambda_p)\right)^{n_p}.
	$$
	Let $L_j := {\rm Par}(\sigma_j) = \Span_\RR\big\lbrace\fu - \fv \mid \fu, \fv \in \sigma_j\big\rbrace$ be the linear space parallel to $\sigma_j$.
	We first claim that
	\begin{equation}
		\label{eq_facet_transversality}
		\deg_{\underline{\Lambda}}^\bn(\Gamma) \;>\;0
		\quad\text{ if and only if }\quad
		L_j+{\rm Par}(\tau)\;=\;N_\RR
	\end{equation}
	for some facet $\sigma_j$ of $\Gamma$ and some facet $\tau$ of $\Delta$.
	By \autoref{thm_AHR}(i) and \autoref{prop_multideg_bdeq}, translating $\Gamma$ does not change the
	left-hand side.

	Assume $\deg_{\underline{\Lambda}}^\bn(\Gamma) >0$.
	Translate $\Gamma$ generically so that it meets $\Delta$
	transversely.
	Their stable intersection is then supported on the nonempty set $\Gamma\cap\Delta$, and the
	facets $\sigma_j$ and $\tau$ containing a transversal intersection point satisfy
	$L_j+{\rm Par}(\tau)=N_\RR$.

	 Conversely, assume that such facets exist. Then
	$
	U:=\Relint(\tau)-\Relint(\sigma_j)
	$
	is a nonempty open Euclidean subset of $N_\RR$. 
	Choose a point $\ba\in U$. 
	The facets $\sigma_j+\ba$ and $\tau$ meet transversely in their relative interiors, and so we get $\deg((\Gamma+\ba)\st\Delta)>0$.
	 Translation invariance gives 	$\deg(\Gamma\st\Delta)>0$, proving the claim \autoref{eq_facet_transversality}.

	Applying the same criterion \autoref{eq_facet_transversality}  to the linear space $L_j$ gives
	$$
	\deg_{\underline{\Lambda}}^\bn(L_j)>0
	\quad\text{ if and only if }\quad
	L_j+{\rm Par}(\tau)=N_\RR
	$$
	for some facet $\tau$ of $\Delta$.
	As a consequence,
	\begin{equation}
		\label{eq_equiv_pos_L_j}
		\deg_{\underline{\Lambda}}^\bn(\Gamma)>0
		\quad\text{ if and only if }\quad
		\deg_{\underline{\Lambda}}^\bn(L_j)>0
	\end{equation}
	for some $1\le j\le k$.

	Set $\Lambda_i'={\rm Rec}(\Lambda_i)$ and
	$\underline{\Lambda'}=\Lambda_1',\ldots,\Lambda_p'$. Each $\Lambda_i'$ is a positive tropical fan
	divisor: since $\Lambda_i\bdeq\Lambda_i'$ by \autoref{thm_AHR}(ii), for every tropical curve $C$,
	\autoref{prop_multideg_bdeq} gives
	$
	\deg(\Lambda_i'\st C)=\deg(\Lambda_i\st C)>0.
	$

	Finally, by applying \autoref{prop_multideg_bdeq} and \autoref{prop_trop_multdeg_L} to each $L_j$, we obtain that
	$$
	\deg_{\underline{\Lambda}}^\bn(L_j) \;=\; \deg_{\underline{\Lambda'}}^\bn(L_j) \;>\; 0
	$$
	if and only if
	$$
		|\bn|_I \;\le\; \dim\left(\Pi_I(L_j)\right) \;=\; \dim\left(\Pi_I(\sigma_j)\right) \qquad \text{ for all $I \subseteq [p]$}.
	$$
	The desired equivalence now follows from \autoref{eq_equiv_pos_L_j}.
	This concludes the proof of the theorem.
\end{proof}

\begin{remark}
	We point out that \autoref{thm_main_crit} can also be deduced from \cite[Lemma~2.9]{He} and \cite[Theorem~1.2]{He} by applying the latter to the linear spaces parallel to the facets of $\Gamma$.
	Our independent method of proof has the advantage of yielding log-concavity results for suitable tropical varieties, as demonstrated in \autoref{sec_augmented} for augmented Bergman fans of polymatroids. 
	Our method proceeds in three steps.
	First, bounded rational equivalence, as developed by Allermann--Hampe--Rau \cite{AHR}, replaces the relevant tropical cycles by their recession fans.
	Second, toric methods, via Fulton--Sturmfels theory of Minkowski weights \cite{FS}, translate the problem into intersection theory.
	Finally, the positivity result from \cite{CCLMZ} supplies the decisive positivity input.	
\end{remark}

From the results already obtained, we readily derive the desired positivity criterion for projection-pure and facet-selectable tropical varieties.

\begin{theorem}
	\label{thm_crit_two_cond}
	Let $\Gamma\subset N_\RR = \RR^{m_1}\times\cdots\times \RR^{m_p}$ be a $d$-dimensional projection-pure and facet-selectable tropical variety.
	Let $\underline{\Lambda}=\Lambda_1,\ldots,\Lambda_p$ where each $\Lambda_i\subset \RR^{m_i}$ is a positive tropical divisor.
	For every $\bn=(n_1,\ldots,n_p)\in\NN^p$ with $|\bn|=d$, we have 
	$$
	\deg_{\underline{\Lambda}}^\bn(\Gamma)>0 \quad \text{ if and only if } \quad \sum_{i \in I} n_i \;\le\; \dim\left(\Pi_I(\Gamma)\right) \text{\; for every \;$I\subseteq[p]$.}
	$$
	Moreover, the function
	$
	r_\Gamma:2^{[p]}\longrightarrow\NN
	$,
	$I\longmapsto\dim\left(\Pi_I(\Gamma)\right)$
	is the rank function of a polymatroid.
\end{theorem}
\begin{proof}
	The positivity criterion follows directly from \autoref{thm_main_crit} and the definition of
	facet-selectability.
	The assertion that $r_\Gamma$ is a polymatroid rank function follows from
	\autoref{prop_rank_function} because $\Gamma$ is projection-pure.
\end{proof}

\section{Non-Lorentzian tropical volume polynomials}
\label{sec_non_lorentzian}

Under the hypotheses of \autoref{thmB}, the support of the tropical volume polynomial is M-convex. 
It is therefore natural to ask whether the same hypotheses impose the Lorentzian condition. 
The answer is no. 
We give two counterexamples, the second of which is translation-admissible and connected in codimension one.

We first recall the notion of Lorentzian polynomials as introduced by Br\"and\'en and Huh \cite{BH}. 
Let $h(t_1,\ldots,t_p)$ be a homogeneous polynomial of degree $d$ in $\RR[\ttt]=\RR[t_1,\ldots,t_p]$.

\begin{definition}[\cite{BH}]
The homogeneous polynomial $h$ is called \emph{Lorentzian} if the following conditions hold:
\begin{enumerate}[\rm (i)]
\item The coefficients of $h$ are nonnegative.
\item The support of $h$ is an M-convex set (equivalently, the set of lattice points of a polymatroid base polytope).
\item The quadratic form
$$
\partial_{t_{i_1}}\partial_{t_{i_2}}\cdots\partial_{t_{i_e}}(h)
$$
has at most one positive eigenvalue for any
$1\leq i_1\leq i_2\leq\cdots\leq i_e\leq p$, where $e=d-2$.
\end{enumerate}
\end{definition}

Let $\Gamma \subset N_\RR = \RR^{m_1} \times \cdots \times \RR^{m_p}$ be a $d$-dimensional tropical variety and $\underline{\Lambda}=\Lambda_1,\ldots,\Lambda_p$ where each $\Lambda_i \subset \RR^{m_i}$ is a positive tropical divisor. 
We encode all the tropical multidegrees in the \emph{tropical volume polynomial}
$$
{\rm tvol}_{\Gamma,\underline\Lambda}(\ttt)
\;:=\;
\sum_{|\bn|=d}\frac{d!}{n_1!\cdots n_p!}\,
\deg_{\underline\Lambda}^{\bn}(\Gamma)t_1^{n_1}\cdots t_p^{n_p}.
$$
By \autoref{thmB}, projection-purity and facet-selectability determine the support of this polynomial.
The following examples show that they do not impose the Lorentzian signature condition on its coefficients.
The computations below repeatedly use the distributive property of stable intersection over the sum of tropical cycles \cite[Remark~2.5(1), Theorem~2.16]{JY}.

\begin{proposition}
\label{prop_non_lorentzian_two_axioms}
There exists a projection-pure, facet-selectable tropical variety that is not translation-admissible and whose tropical volume polynomial, with respect to some sequence of positive tropical divisors, is not Lorentzian.
\end{proposition}
\begin{proof}
We consider the tropical variety of \autoref{ex_two_conditions_not_translation}.
Thus $\Gamma \subset N_\RR = \RR \times \RR \times \RR^2$ is the tropical cycle given as the sum of the rational linear spaces
$$
L_1\;:=\;\Span_\RR(\bv_1,\bv_2),
\qquad
\bv_1\;=\;(1,0,1,0),\quad \bv_2\;=\;(0,1,0,1),
$$
and
$$
L_2\;:=\;\Span_\RR(\bu_1,\bu_2),
\qquad
\bu_1\;=\;(2,0,1,0),\quad \bu_2\;=\;(0,2,0,1),
$$
both with weight one.
 Let $\Lambda_1=\{0\}\subset\RR$ and $\Lambda_2=\{0\}\subset\RR$.
In the third factor $\RR^2$, set $
\Lambda_3=\bigl(\{0\}\times\RR\bigr)+\bigl(\RR\times\{0\}\bigr) \subset \RR^2;
$
this is a positive tropical divisor: its intersection number with a tropical curve is the sum of the
degrees of the two coordinate projections, at least one of which is positive. 
Set $\underline{\Lambda} = \Lambda_1, \Lambda_2, \Lambda_3$.
It remains to compute the corresponding tropical volume polynomial.

Consider $N_\RR$ as the product of four real lines by decomposing the third factor:
$N_\RR=\RR\times\RR\times\RR^2=\left(\RR^1\right)^4$. 
For $1\leq i\leq4$, let
$D_i\subset N_\RR$ be the preimage of the origin $\{0\} \subset \RR$ in the $i$-th factor.
Set
$$
d_{ij}\;:=\;\deg(L_1\st D_i\st D_j)
\qquad\text{ and }\qquad
e_{ij}\;:=\;\deg(L_2\st D_i\st D_j).
$$
Notice that $d_{ii}=e_{ii}=0$. 
In terms of the absolute values of the corresponding coordinate minors, we have
{\footnotesize
$$
d_{12}=\left|\left|\begin{smallmatrix}1&0\\0&1\end{smallmatrix}\right|\right|=1,\quad
d_{13}=\left|\left|\begin{smallmatrix}1&1\\0&0\end{smallmatrix}\right|\right|=0,\quad
d_{14}=\left|\left|\begin{smallmatrix}1&0\\0&1\end{smallmatrix}\right|\right|=1,\quad
d_{23}=\left|\left|\begin{smallmatrix}0&1\\1&0\end{smallmatrix}\right|\right|=1,\quad
d_{24}=\left|\left|\begin{smallmatrix}0&0\\1&1\end{smallmatrix}\right|\right|=0,\quad
d_{34}=\left|\left|\begin{smallmatrix}1&0\\0&1\end{smallmatrix}\right|\right|=1
$$
}
and 
{\footnotesize
$$
e_{12}=\left|\left|\begin{smallmatrix}2&0\\0&2\end{smallmatrix}\right|\right|=4,\quad
e_{13}=\left|\left|\begin{smallmatrix}2&1\\0&0\end{smallmatrix}\right|\right|=0,\quad
e_{14}=\left|\left|\begin{smallmatrix}2&0\\0&1\end{smallmatrix}\right|\right|=2,\quad
e_{23}=\left|\left|\begin{smallmatrix}0&1\\2&0\end{smallmatrix}\right|\right|=2,\quad
e_{24}=\left|\left|\begin{smallmatrix}0&0\\2&1\end{smallmatrix}\right|\right|=0,\quad
e_{34}=\left|\left|\begin{smallmatrix}1&0\\0&1\end{smallmatrix}\right|\right|=1.
$$
}
The divisor coming from $\Lambda_3\subset\RR^2$ pulls back to $D_3+D_4$. 
Therefore, we obtain
\begin{align*}
	\deg_{\underline\Lambda}^{(1,1,0)}(\Gamma) &\;=\; \deg(\Gamma\st D_1\st D_2) \;=\; d_{12}+e_{12} \;=\; 5,\\
	\deg_{\underline\Lambda}^{(1,0,1)}(\Gamma) &\;=\; \deg\bigl(\Gamma\st D_1\st(D_3+D_4)\bigr)
	\;=\; d_{13}+d_{14}+e_{13}+e_{14} \;=\; 3, \\
	\deg_{\underline\Lambda}^{(0,1,1)}(\Gamma) &\;=\; \deg\bigl(\Gamma\st D_2\st(D_3+D_4)\bigr)
	\;=\; d_{23}+d_{24}+e_{23}+e_{24} \;=\; 3,\\
	\deg_{\underline\Lambda}^{(0,0,2)}(\Gamma) &\;=\; \deg\bigl(\Gamma\st(D_3+D_4)^2\bigr) \;=\; 2(d_{34}+e_{34}) \;=\; 4.
\end{align*}
As a consequence,
$$
{\rm tvol}_{\Gamma,\underline\Lambda}(t_1,t_2,t_3)
\;=\; 10t_1t_2+6t_1t_3+6t_2t_3+4t_3^2,
$$
whose Hessian is
$$
A\;=\;
\begin{pmatrix}
0&10&6\\
10&0&6\\
6&6&8
\end{pmatrix}.
$$
Its characteristic polynomial is $\chi_A(\lambda)=(\lambda+10)(\lambda^2-18\lambda+8)$.
Thus the eigenvalues are
$$
-10,\qquad 9-\sqrt{73},\qquad 9+\sqrt{73}.
$$
The last two are positive, and so ${\rm tvol}_{\Gamma,\underline\Lambda}$ is not Lorentzian.
\end{proof}

The next proposition shows that this failure persists for translation-admissible tropical varieties
connected in codimension one.

\begin{proposition}
\label{prop_non_lorentzian_translation}
There exists a translation-admissible tropical variety that is connected in codimension one and whose tropical volume polynomial, with respect to some sequence of positive tropical divisors, is not Lorentzian.
\end{proposition}
\begin{proof}
Write
$
N_\RR=\RR^2\times\RR^2$ with standard basis $\ee_1,\ldots,\ee_4 \in \NN^4$. 
For $1\leq i<j\leq4$, let $L_{ij}:=\Span_\RR(\ee_i,\ee_j)$. 
Consider the tropical cycle $\Gamma \subset N_\RR$ given as the sum
$$
\Gamma \;=\; 5L_{12} + L_{13} + L_{14} + L_{23} + L_{24} + 5L_{34}.
$$
The support of $\Gamma$ is equal to $S:=|\Gamma|=\bigcup_{1\leq i<j\leq4}L_{ij}$.

We also see $N_\RR = \left(\RR^1\right)^4$ as a product of four real lines. 
For $1\leq i\leq4$, let
$D_i\subset N_\RR$ be the preimage of the origin $\{0\} \subset \RR$ in the $i$-th factor.
Set $\PP = \left(\PP^1\right)^4$ with coordinates $[x_1:y_1] \times  [x_2:y_2] \times [x_3:y_3] \times [x_4:y_4]$.
Let $f$ and $g$ be two general elements of $\HH^0\!\left(\PP, \OO_{\PP}(1,1,1,1)\right)$.
Let $\KK = \CC\{\{t\}\}=\bigcup_{q \ge 1}\CC((t^{1/q}))$ be the field of Puiseux series with its natural valuation ${\rm val} : \KK \rightarrow \QQ \cup \{\infty\}$.
Let $T = (\KK^*)^4$ be the natural torus on $\PP_\KK = \PP \times_\CC \KK = \left(\PP_\KK^1\right)^4$.
Consider the irreducible varieties $Y_1 = \left(V(f) \times_\CC \KK\right) \cap T$ and $Y_2 = \left(V(g) \times_\CC \KK\right) \cap T$.
By writing $T = \Spec(\KK[z_1^{\pm 1}, z_2^{\pm 1}, z_3^{\pm 1}, z_4^{\pm 1}])$ with $z_i = x_i/y_i$, we get 
$$
Y_1 = V\Bigg(\sum_{0 \le n_i \le 1} a_\bn\, z_1^{n_1}z_2^{n_2}z_3^{n_3}z_4^{n_4}\Bigg) \quad \text{ and } \quad Y_2 = V\Bigg(\sum_{0 \le n_i \le 1} b_\bn\, z_1^{n_1}z_2^{n_2}z_3^{n_3}z_4^{n_4}\Bigg),
$$
where the coefficients $a_\bn, b_\bn \in \CC \subset \KK$ are general. 
The tropicalizations of $Y_1$ and $Y_2$ are equal to 
$$
\Trop(Y_1) \;=\; \Trop(Y_2) \;=\; D_1 + D_2 + D_3 + D_4
$$
(see \cite[Proposition 3.1.10]{MS}).
Notice that $D_i^2=0$ and $D_i\st D_j=L_{k\ell}$ when
$\{k,\ell\}=[4]\setminus\{i,j\}$.
Let $\bu = (u_1,u_2,u_3,u_4) \in T$ be a general element with ${\rm val}(u_i) = 0$.
Consider 
$$
\bu^{-1}\cdot g \;:=\; g(u_1^{-1}x_1,y_1,u_2^{-1}x_2,y_2,u_3^{-1}x_3,y_3,u_4^{-1}x_4,y_4) \;\in\; \HH^0\!\left(\PP_\KK, \OO_{\PP_\KK}(1,1,1,1)\right).
$$
By utilizing Bertini's theorem, we get that 
$$
Y \;=\; Y_1 \cap \bu Y_2 \;=\; V\big(f,\bu^{-1}\cdot g \big) \cap T
$$ is an irreducible variety. 
Then \cite[Theorem~3.6.1]{MS} gives
$$
\Trop(Y) \;=\; \Trop(Y_1 \cap \bu Y_2) \;=\;  \left(D_1+D_2+D_3+D_4\right)^2 \;=\; 2\sum_{1\leq i<j\leq4}L_{ij}.
$$
In particular, $|\Trop(Y)|=S$. 
The tropicalization $\Trop(Y)$ is connected in codimension one by \cite[Theorem~3.5.1]{MS} and translation-admissible by \cite[Lemma~2.13]{He}. 
Both properties depend only on the support, and so $\Gamma$ is connected in codimension one and translation-admissible.

It remains to compute a non-Lorentzian tropical volume polynomial.
 Let $\mathcal{H}_1 \subset\RR^2$ and $\mathcal{H}_2 \subset \RR^2$ be the standard tropical lines in the first and second factor of $N_\RR = \RR^2 \times \RR^2$, respectively (see \autoref{rem_std_trop_hyper}). 
 Let $\underline{\mathcal{H}}=\mathcal{H}_1,\mathcal{H}_2$. 
 We can compute that 
$$
\deg_{\underline{\mathcal{H}}}^{(2,0)}(\Gamma)\;=\;5,\qquad \deg_{\underline{\mathcal{H}}}^{(1,1)}(\Gamma)\;=\;4,\qquad \deg_{\underline{\mathcal{H}}}^{(0,2)}(\Gamma)\;=\;5.
$$
As a consequence,
$$
{\rm tvol}_{\Gamma,\underline{\mathcal{H}}}(t_1,t_2)
=5t_1^2+8t_1t_2+5t_2^2,
$$
and the corresponding Hessian is
$$
A \;=\;
\begin{pmatrix}
10&8\\
8&10
\end{pmatrix}.
$$
The eigenvalues of $A$ are $2$ and $18$. 
Therefore, the tropical volume polynomial ${\rm tvol}_{\Gamma,\underline{\mathcal{H}}}(t_1,t_2)$ is not Lorentzian.
\end{proof}

\section{Augmented Bergman fans and Lorentzian tropical volume polynomials}
\label{sec_augmented}

In this section, we study the tropical multidegrees of the augmented Bergman fan of a polymatroid. 
Contrary to the counterexamples given in \autoref{sec_non_lorentzian}, we show that in the case of augmented Bergman fans we do get Lorentzian tropical volume polynomials.

Let $\sP$ be a polymatroid on $[p]$ with cage $\bm=(m_1,\ldots,m_p) \in \ZZ_+^p$.
Let $m = m_1 + \cdots + m_p$.
 Choose pairwise disjoint sets $E_i$ with $|E_i|=m_i$, set $E:=\bigsqcup_iE_i$, and let $\pi:E\to[p]$ be the map with fibers $\pi^{-1}(i)=E_i$.
A subset $F\subseteq[p]$ is a \emph{flat} if $\rk_\sP(F\cup\{i\})>\rk_\sP(F)$ for every $i\notin F$; it is \emph{proper} if $F\neq[p]$.
For a finite set $A$, write $\RR^A:=\prod_{i\in A}\RR$, let $\big\{\ee_i\big\}_{i\in A}$ be its standard basis, and set $\ee_S:=\sum_{i\in S}\ee_i$ for all $S \subseteq A$.

\begin{definition}[{Augmented Bergman fan of a polymatroid; Eur-Larson \cite[Definition~3.7]{ELP}}]
	A subset $S \subseteq E$ and a flag $\mathscr{F} = \big\lbrace F_1 \subsetneq F_2 \subsetneq \cdots \subsetneq F_k  \big\rbrace$ of proper flats of $\sP$ are said to form a \emph{compatible pair}, written
	$S\leq\mathscr{F}$, if the following two conditions are satisfied:
	\begin{enumerate}[\rm\;\; (a)]
		\item\label{cond_a} $\rk_\sP(\pi(T))\geq |T|$ for every $T\subseteq S$.
		\item\label{cond_b} $\rk_\sP(F\cup\pi(T))>\rk_\sP(F)+|T|$ for every $F\in\mathscr{F}$ and every nonempty
		$T\subseteq S\setminus\pi^{-1}(F)$.
	\end{enumerate}
	For a compatible pair $S\leq\mathscr{F}$, we consider the cone 
	$$
	\sigma_{S\leq\mathscr{F}}
	\;:=\;
	\Cone\big(\ee_s \mid s\in S\big)
	\;+\;
	\Cone\big(-\ee_{E\setminus\pi^{-1}(F)} \mid F\in\mathscr{F} \big).
	$$ 
	The \emph{augmented Bergman fan} of $\sP$ is
	$$
		\Sigma_\sP
		\;:=\;
		\big\{
		\sigma_{S\leq\mathscr{F}}
		\;\mid\;
		S\leq\mathscr{F}\text{ compatible}
		\big\}
		\;\subseteq\; \RR^E,
	$$
	with weight one on every facet.
\end{definition}

Let $\mathscr B_\bm$ be the Boolean polymatroid on $[p]$ with cage $\bm$, whose rank function is
$
	\rk_{\mathscr B_\bm}(I):=\sum_{i\in I}m_i
$
for all $I \subseteq [p]$.
The augmented Bergman fan of $\mathscr B_\bm$ is the \emph{polystellahedral fan} 
$
	\Sigma_\bm := \Sigma_{\mathscr B_\bm}
$
with cage $\bm$.
This is a smooth projective fan, and we write $X_\bm:=X_{\Sigma_\bm}$ for its associated \emph{polystellahedral toric variety} (see \cite[Proposition~2.2, Proposition~2.3]{ELP}).

In the remark below, we recall some of the basic properties of the augmented Bergman fan of a polymatroid. 

\begin{remark}
	Let $\sP$ be a rank-$r$ polymatroid with cage $\bm=(m_1,\ldots,m_p)$.
	The fan $\Sigma_\sP$ is a pure balanced subfan of the polystellahedral fan $\Sigma_\bm$ (see \cite[Theorem 3.8]{ELP}).
	The support of $\Sigma_\sP$ equals the support of the augmented Bergman fan of the matroid given as the multisymmetric lift $M_\pi(\sP)$ (see \cite[Theorem~3.8]{ELP}). 
	In particular, $\dim\left(\Sigma_\sP\right)=\rk(\sP)$. 
	The zero extension of $\Sigma_\sP$ to $\Sigma_\bm$ defines the augmented Bergman class $\left[\Sigma_\sP\right] \in A_{r}\left(X_{\Sigma_\bm}\right)$ (see \cite[Definition~3.12]{ELP}).
\end{remark}

The following remark provides the basic motivation for considering the \emph{augmented} Bergman fan of a polymatroid, rather than the Bergman fan of a polymatroid. 

\begin{remark}
	Let $\sP$ be a rank-$r$ polymatroid with cage $\bm=(m_1,\ldots,m_p)$.
	 The equality $\dim\left(\Sigma_\sP\right)=r$ gives the correct dimension for the multidegrees considered below.
	  Suppose that $\sP$ is realized over  $\CC$ by an $r$-dimensional subspace $L\subseteq\bigoplus_iV_i$, where $\dim_\CC(V_i)=m_i$ and
	$\rk_\sP(I)=\dim_\CC\left(\IM\left(L\to\bigoplus_{i\in I}V_i\right)\right)$ for every $I\subseteq[p]$. 
	The \emph{augmented wonderful compactification} $W_L$ of $L$ is given as the closure 
	$$
	W_L \;:=\; \overline{\;\IM\left(L \;\longrightarrow\;
	\prod_{\varnothing \subsetneq S\subseteq[p]}
	\PP\!\left(\bigoplus_{i\in S}V_i\oplus\CC\right)\right)\;}.
	$$
	We have that $W_L$ embeds in the corresponding polystellahedral variety $X_\bm$ and that $\left[W_L\right] \;=\; \left[\Sigma_\sP\right]$ in $A_r(X_\bm)$ (see \cite[Definition~1.2 and Proposition~3.20]{ELP}). 
	By projecting to the singleton factors, we obtain a proper birational morphism $W_L\to Y_L$, where $Y_L$ is the closure of the image of the subspace $L$ in $\prod_i\PP(V_i\oplus\CC) = \PP^{m_1} \times \cdots \times \PP^{m_p}=:\PP$.
	 Finally, the support of positive multidegrees of $Y_L$ is given by 
	 $$
	 \msupp_\PP\left(Y_L\right) \;=\; B(\sP) \cap \NN^p
	 $$
	 (see \cite[Proposition 5.4]{CCLMZ}, \cite[Proposition 7.15]{CCRMM}).
\end{remark}

The main result of this section is the following. 

\begin{theorem}
	\label{thm_augmented_main}
	Let $\sP$ be a rank-$r$ polymatroid on $[p]$ with cage
	$\bm=(m_1,\ldots,m_p) \in \ZZ_+^p$.
	Let
	$\underline\Lambda=\Lambda_1,\ldots,\Lambda_p$ where each
	$\Lambda_i\subset\RR^{E_i}$ is a positive tropical divisor. 
	Then:
	\begin{enumerate}[\rm (i)]
		\item $\Sigma_\sP$ is projection-pure and facet-selectable.
		\item For every $\bn=(n_1,\ldots,n_p)\in\NN^p$ with $|\bn|=r$, we have
		$$
		\deg_{\underline{\Lambda}}^\bn(\Sigma_\sP)>0 \quad \text{ if and only if } \quad \bn \in B(\sP).
	$$
	\item The tropical volume polynomial
	$$
	{\rm tvol}_{\Sigma_\sP, \underline{\Lambda}}(\ttt) \;=\; \sum_{|\bn| = r} \frac{r!}{n_1!\cdots n_p!}  \, \deg_{\underline{\Lambda}}^\bn(\Sigma_\sP) \, t_1^{n_1} \cdots t_p^{n_p}
	$$
	is Lorentzian.
	\end{enumerate}
\end{theorem}

Before proving this theorem, we need the following technical lemma, which shows that augmented Bergman fans satisfy the two conditions that we introduced (i.e., projection-purity and facet-selectability).

\begin{lemma}
	\label{lem_aug_projection_properties}
	Let $\sP$ be as in \autoref{thm_augmented_main}.
	For each $I\subseteq[p]$, let $\sP|I$ denote the polymatroid on $I$ with cage $\big(m_i\big)_{i\in I}$ and rank function
	$\rk_{\sP|I}(A):=\rk_\sP(A)$ for all $A \subseteq I$.
	Then:
	\begin{enumerate}[\rm (i)]
		\item $\Pi_I\left(\left|\Sigma_\sP\right|\right) = \left|\Sigma_{\sP|I}\right|$.
		\item $\dim\left(\Pi_I\left(\Sigma_\sP\right)\right) = \rk_\sP(I)$.
		\item $\Sigma_\sP$ is projection-pure and facet-selectable.
	\end{enumerate} 
\end{lemma}
\begin{proof}
	(i) 
	This part could be obtained by combining \cite[Theorem 3.8]{ELP} and \cite[Proposition 3.1]{BHMPW}.
	However, for completeness we give a short and self-contained proof. 
	
	We repeatedly use the  diminishing-returns inequality
	$$
	\rk_\sP(A\cup C)-\rk_\sP(A)
	\;\geq\; \rk_\sP(B\cup C)-\rk_\sP(B)
	$$
	for all $A\subseteq B\subseteq[p]$ and
	$C\subseteq[p]$.
	Fix $I\subseteq[p]$, put $E_I:=\bigsqcup_{i\in I}E_i$, and let
	$\pi_I:=\pi|_{E_I}$. 
	Let $S\leq\mathscr{F}$ be a compatible pair of $\sP$. 
	We set $S_I:=S\cap E_I$, and we obtain $\mathscr{F}_I$ from $\big\lbrace F\cap I \big\rbrace_{F\in\mathscr{F}}$ by deleting repetitions and the instances where $F \cap I = I$.
	We claim that each member $F\cap I$ of $\mathscr{F}_I$ is a proper flat of $\sP|I$; indeed, if $i\in I\setminus F$, then
	$$
		\rk_\sP\left((F\cap I)\cup\{i\}\right)-\rk_\sP\left(F\cap I\right)
		\;\geq\; \rk_\sP\left(F\cup\{i\}\right)-\rk_\sP\left(F\right)>0.
	$$
	The pair $S_I\leq\mathscr{F}_I$ inherits the condition~\hyperref[cond_a]{(a)}. 
	For condition~\hyperref[cond_b]{(b)}, let $A=F\cap I$ and $\varnothing\neq T\subseteq S_I\setminus\pi_I^{-1}(A)$,  then
	$T\subseteq S\setminus\pi^{-1}(F)$ and so we obtain
	$$
		\rk_\sP\left(A\cup\pi_I(T)\right)-\rk_\sP\left(A\right)
		\;\geq\; \rk_\sP\left(F\cup\pi(T)\right)-\rk_\sP\left(F\right) \;>\; |T|.
	$$
	Therefore $S_I\leq\mathscr{F}_I$ is a compatible pair of $\sP|I$. 
	The projection $\Pi_I : N_\RR \rightarrow \RR^{E_I}$ sends $\ee_s$ to $\ee_s$ or $0$
	according to whether $s\in E_I$ or not, and sends
	$-\ee_{E\setminus\pi^{-1}(F)}$ to
	$-\ee_{E_I\setminus\pi_I^{-1}(F\cap I)}$. After zero vectors and duplicate generators
	are removed, these are precisely the generators of $\sigma_{S_I\leq\mathscr{F}_I}$, and thus we obtain
	$\Pi_I\left(\sigma_{S\leq\mathscr{F}}\right)=\sigma_{S_I\leq\mathscr{F}_I}$.
	This shows the inclusion
	$\Pi_I(|\Sigma_\sP|)\subseteq|\Sigma_{\sP|I}|$.

	We now show the reverse inclusion. 
	Recall 
	$\cl_{\sP}(A):=\left\{i\in[p]\mid
	\rk_\sP(A\cup\{i\})=\rk_\sP(A)\right\}$. 
	Let $S'\leq\mathscr{G}$ be a compatible pair of $\sP|I$. 
	Since every $G\in\mathscr{G}$ is a flat of $\sP|I$, we have $\operatorname{cl}_{\sP}(G)\cap I=G$.
	Hence we get that
	$\widehat{\mathscr{G}}:=\big\lbrace\operatorname{cl}_{\sP}(G)\big\rbrace_{G\in\mathscr{G}}$ is a strict chain of proper
	flats of $\sP$. 
	The condition~\hyperref[cond_a]{(a)} for $S'\leq\widehat{\mathscr{G}}$ is inherited. If
	$G\in\mathscr{G}$ and
	$\varnothing\neq T\subseteq S'\setminus\pi^{-1}\left(\operatorname{cl}_{\sP}(G)\right)$, then
	$T\subseteq S'\setminus\pi_I^{-1}(G)$ and
	$$
		\rk_\sP\left(\operatorname{cl}_{\sP}(G)\cup\pi(T)\right)
		-\rk_\sP\left(\operatorname{cl}_{\sP}(G)\right)
		\;=\; \rk_\sP\left(G\cup\pi_I(T)\right)-\rk_\sP\left(G\right) \;>\; |T|.
	$$
	Thus condition~\hyperref[cond_b]{(b)} also holds. 
	Moreover,
	$\Pi_I\left(-\ee_{E\setminus\pi^{-1}(\operatorname{cl}_{\sP}(G))}\right) =-\ee_{E_I\setminus\pi_I^{-1}(G)}$, and so the lifted cone projects to $\sigma_{S'\leq\mathscr{G}}$. 
	This proves the reverse inclusion $\Pi_I(|\Sigma_\sP|)\supseteq|\Sigma_{\sP|I}|$.

\smallskip

	(ii) By part (i), we obtain $\dim\left(\Pi_I(\Sigma_\sP)\right)=\dim\left(\Sigma_{\sP|I}\right)=\rk_{\sP|I}(I)=\rk_\sP(I)$.

\smallskip

	(iii) Since $\Sigma_{\sP|I}$ is a tropical fan, part~(i) also proves that
	$\Sigma_\sP$ is projection-pure.
	It remains to prove that $\Sigma_\sP$ is facet-selectable.
	 By part (ii), let $\bn = (n_1,\ldots,n_p)\in\NN^p$ such that
	$|\bn| = \dim(\Sigma_\sP) = \rk_\sP([p])$ and $|\bn|_I\leq  \dim(\Pi_I(\Sigma_\sP)) = \rk_\sP(I)$ for all $I\subseteq[p]$. 
	Since $n_i\leq\rk_\sP(\{i\})\leq m_i$, we can choose $B_i\subseteq E_i$ with $|B_i|=n_i$, and set
	$B:=\bigsqcup_iB_i$. 
	For every $T\subseteq B$, we obtain
	$$
		\left|T\right| \;\leq\; \sum_{i\in\pi(T)} \left|B_i\right| \;=\; \left|\bn\right|_{\pi(T)} \;\leq\; \rk_\sP(\pi(T)).
	$$
	Thus $B\leq\varnothing$ is a compatible pair of $\sP$: condition~\hyperref[cond_a]{(a)} holds, while condition~\hyperref[cond_b]{(b)} is vacuous.
	The coordinate cone $\sigma_{B \le \varnothing}:=\cone(\ee_s)_{s\in B}$ has dimension
	$|B|=\rk_\sP([p])=\dim\left(\Sigma_\sP\right)$, and so it is a facet. 
	Finally, we have
	$$
	\dim\left(\Pi_I(\sigma_{B \le \varnothing})\right) \;=\; \left|B\cap\pi^{-1}(I)\right| \;=\; \left|\bn\right|_I
	$$ 
	for all $I\subseteq[p]$.
	Therefore $\Sigma_\sP$ is facet-selectable, and so the proof of the lemma is complete. 
\end{proof}

We are now ready to prove our main result regarding augmented Bergman fans.

\begin{proof}[Proof of \autoref{thm_augmented_main}]
	Parts (i) and (ii) follow from  \autoref{lem_aug_projection_properties} 
	and \autoref{thm_crit_two_cond}.

	We now prove part (iii).
	Here we use the same ideas as in \autoref{sect_toric} and transfer tropical varieties and stable intersections to the setting of toric methods and Minkowski weights. 
	Then the result will be a formal consequence of \cite[Theorem 8.3]{EHL}, \cite[Theorem 3.17]{ELP} and \cite[Theorem 1.6]{ADH}.
	
	By \autoref{thm_AHR} and \autoref{prop_multideg_bdeq}, we may assume that each $\Lambda_i \subset \RR^{E_i}$ is a positive tropical fan divisor.
	After applying \autoref{lem_refine_div}, we obtain a smooth complete fan $\Sigma_i \subset \RR^{E_i}$ carrying
	$\Lambda_i$ and an associated basepoint-free torus-invariant Cartier divisor $D_i$ on $X_{\Sigma_i}$. 
    A common complete refinement of $\Sigma_\bm$ and $\Sigma_1 \times \cdots \times \Sigma_p$ admits a smooth projective refinement $\widehat\Sigma$ (see \cite[Theorem~6.1.18 and Theorem~11.1.9]{CLS}). 
    Let $\widehat{X} := X_{\widehat{\Sigma}}$ and $\widehat{D}_i$ be the pullback of $D_i$ to $\widehat{X}$.
    Let $\widehat\Sigma_\sP$ be the induced refinement of $\Sigma_\sP$.
    By utilizing the Fulton-Sturmfels correspondence \cite{FS}, we obtain
	$$
		{\rm tvol}_{\Sigma_\sP,\underline\Lambda}(\ttt)
		\;=\; \int_{\widehat{X}}
			\left(\widehat{D}_1t_1+\cdots+\widehat{D}_pt_p\right)^r \cdot \big[\widehat\Sigma_\sP\big].
	$$
	By \cite[Theorem 3.17]{ELP} and \cite[Theorem 1.6]{ADH}, $\widehat{\Sigma}_\sP$ is a Lefschetz fan. 
	Finally, \cite[Theorem 8.3]{EHL} shows that the tropical volume polynomial ${\rm tvol}_{\Sigma_\sP,\underline\Lambda}(\ttt)$ is Lorentzian. 
\end{proof}

For standard tropical hyperplanes, the following corollary recovers \cite[Corollary~1.4]{ELP} in a tropical form.
Our proof relies on \autoref{thm_augmented_main} which determines positivity and then on the well-known fact that the stable intersection of tropical linear spaces is either empty or again a tropical linear space (see \cite{Speyer,Hampe,EHL,ELP}).

\begin{corollary}[{Standard tropical hyperplanes; Eur--Larson \cite[Corollary 1.4]{ELP}}]
	\label{cor_augmented_standard_hyperplanes}
	In the setting of \autoref{thm_augmented_main}, let $\mathcal H_i\subset\RR^{E_i}$ be the standard tropical hyperplane and set $\underline{\mathcal H}:=\mathcal H_1,\ldots,\mathcal H_p$. 
	For
	$\bn\in\NN^p$ with $|\bn|=r$, we have
	$$
		\deg_{\underline{\mathcal H}}^\bn(\Sigma_\sP)
		\;=\;
		\begin{cases}
			1& \quad \text{if \;\;}\bn\in B(\sP),\\
			0& \quad \text{otherwise}.
		\end{cases}
	$$
	In particular, the corresponding tropical volume polynomial is given by 
	$$
		{\rm tvol}_{\Sigma_\sP,\underline{\mathcal H}}(\ttt)
		\;=\;
		r!\sum_{\bn\in B(\sP)\cap\NN^p}\frac{t_1^{n_1}\cdots t_p^{n_p}}{n_1!\cdots n_p!}.
	$$
\end{corollary}
\begin{proof}
	By \autoref{thm_augmented_main}, we already have that $\deg_{\underline{\mathcal H}}^\bn(\Sigma_\sP) > 0$ if and only if $\bn \in B(\sP)$.
	Therefore it suffices to show that we always have $\deg_{\underline{\mathcal H}}^\bn(\Sigma_\sP) \in \{0, 1\}$.
	For each $i\in[p]$, let $\mathscr{Q}_i$ be the polymatroid with cage $\bm$ and rank function
	$$
	\rk_{\mathscr{Q}_i}(I) \;:=\; \sum_{j \in I} \left(m_j-\delta_{i,j}\right),
	$$ 
	where $\delta_{i,j}$ is the Kronecker delta.
	As tropical cycles, we have that the augmented Bergman fan of $\mathscr{Q}_i$ equals the tropical fan divisor $\Pi_i^{-1}(\mathcal{H}_i) \subset \RR^E$.
	Indeed, by \cite[Theorem 3.8]{ELP}, $\Sigma_{\mathscr{Q}_i}$ is a coarsening of  $\Sigma_{M_\pi(\mathscr{Q}_i)}$, and since the rank function of the multisymmetric lift $M_\pi(\mathscr{Q}_i)$ is given by 
	$$
	\rk_{M_\pi(\mathscr{Q}_i)}(S) \;=\; \min\big\lbrace \left|S\setminus \pi^{-1}(I)\right| + \rk_{\mathscr{Q}_i}(I) \;\mid\; I \subseteq [p]\big\rbrace \;=\; 
	\begin{cases}
		|S|-1 & \quad \text{if } E_i \subseteq S,\\
		|S| & \quad \text{otherwise,}
	\end{cases}
	$$
	it follows that $E_i$ is the unique circuit of $M_\pi(\mathscr{Q}_i)$.
	Therefore, under the Fulton--Sturmfels correspondence \cite{FS}, we obtain
	$$
	\deg_{\underline{\mathcal H}}^\bn(\Sigma_\sP)
	\;=\;
	\int_{X_\bm}
		\big[\Sigma_{\mathscr Q_1}\big]^{n_1}\cdots
		\big[\Sigma_{\mathscr Q_p}\big]^{n_p} \cdot \big[\Sigma_\sP\big].
	$$
	By \cite[Theorem~3.14]{ELP}, the product of two augmented Bergman classes with cage $\bm$ is either zero or, with coefficient one, another augmented Bergman class.
	Since the augmented Bergman fan of the unique rank-zero polymatroid is equal to the origin with weight one, it follows that $\deg_{\underline{\mathcal H}}^\bn(\Sigma_\sP)\in\{0,1\}$.
\end{proof}

\section*{Acknowledgments}

Motivation to bring this paper to completion came from a reading seminar on \emph{``Tropical Varieties''}, co-organized by  Laura Colmenarejo, Jacob Matherne, and the author of this paper, as part of {\it ALCO: A Learning Community} at NC State University.
We thank all the participants of this reading seminar. 

We thank June Huh for helpful comments and for pointing out relevant references.  

The author received support from NSF grant DMS-2502321 and Simons Foundation Travel Support for Mathematicians Award MPS-TSM-00013551.

	\bibliographystyle{amsalpha}
	\bibliography{references}

\end{document}